\documentclass[11pt,a4paper,reqno]{amsart}

\usepackage[utf8]{inputenc}
\usepackage[T1]{fontenc}

\usepackage{amsrefs,amsfonts,amssymb,amsmath,amsthm}

\usepackage[hidelinks]{hyperref}
\usepackage[capitalise,noabbrev]{cleveref}

\usepackage{tikz}
\usetikzlibrary{shapes.geometric,calc,positioning}

\newtheorem{theorem}{Theorem}[section]
\newtheorem{lemma}[theorem]{Lemma}
\newtheorem{proposition}[theorem]{Proposition}
\newtheorem{definition}[theorem]{Definition}
\newtheorem{remark}[theorem]{Remark}
\newtheorem{corollary}{Corollary}[theorem]
\newtheorem{example}[theorem]{Example}
\newtheorem{algorithm}[theorem]{Algorithm}

\title{Geodesic growth in virtually abelian groups}

\author{Alex Bishop}
\address{University of Technology Sydney, Australia}
\urladdr{\url{https://alexbishop.github.io}}
\email{\href{mailto:alexbishop1234@gmail.com}{alexbishop1234@gmail.com}}

\begin{document}
\begin{abstract}
	We show that the geodesic growth function of any finitely generated virtually abelian group is either polynomial or exponential; and that the geodesic growth series is holonomic, and rational in the polynomial growth case.
	In addition, we show that the language of geodesics is blind multicounter.
	
	\bigskip
	
	
\end{abstract}
\maketitle

\section{Introduction}\label{sec:introduction}

The concept of growth in groups is well studied, with famous results including Gromov's classification of groups with polynomial growth~\cite{gromov1981}, and Grigorchuk's example of intermediate growth~\cite{grigorchuk1983} which initiated a great interest in such groups~\cite{mann2012,nekrashevych2005,harpe2000,bartholdi2003a}.

Bridson \textit{et al.}~\cite{bridson2012} asked if there exists a group with intermediate \textit{geodesic} growth, and if there is a characterisation of groups with polynomial \textit{geodesic} growth.
Towards these questions they showed that there is no nilpotent group with intermediate geodesic growth, and provided a sufficient condition for a virtually abelian group to have polynomial geodesic growth.
In this paper, we take the next step in the study of this problem by showing that no virtually abelian group can have intermediate geodesic growth with respect to any finite (weighted monoid) generating set.

The study of geodesic growth of abelian groups began with Shapiro~\cite{shapiro1997} who considered the function $p_S\colon G \to \mathbb{N}$ which counts the geodesics corresponding to a given element of an abelian group $G$ with respect to a generating set $S$.

Benson~\cite{benson1983} showed that the standard growth series for virtually abelian groups is rational with respect to any finite (weighted monoid) generating set.
This result was recently generalised by Evetts~\cite{evetts2019} who showed that the coset, subgroup, and conjugacy growth series of a virtually abelian group is rational with respect to any finite (weighted monoid) generating set.
In this paper, we combine the work of Benson~\cite{benson1983} with a result by Massazza~\cite{massazza1993} on languages with holonomic generating functions. 
In particular, we construct a weight-preserving bijection from the set of all geodesics in a virtually abelian group to a formal language with holonomic generating function.
Then, using the properties of holonomic functions, we show that the geodesic growth of a virtually abelian group is either polynomial or exponential.
However, this bijection is not a monoid homomorphism, and cannot be used to immediately obtain a formal language characterisation for the set of all geodesics in a virtually abelian group.
Instead, we prove that the language of geodesics is accepted by a \emph{blind multicounter automaton}, as defined by Greibach~\cite{greibach1978}.
Thus, this paper also builds on the known results of formal language classification of geodesics \cite{gromov1987,epstein1992,cleary2006,elder2005}.

We begin by establishing our notation in \cref{sec:notation}.
Then, in \cref{sec:holonomic-functions} we provide an overview of the theory of holonomic functions, which we use in \cref{sec:polyhedrally-constrained-languages} where we introduce a family of languages with holonomic generating functions.
In \cref{sec:patterned-words}, we define patterned words in virtually abelian groups, and construct an algorithm to `shuffle' words into this form.
In \cref{sec:geodesic-growth} we show that there is no virtually abelian group with intermediate geodesic growth, in particular, we make use of the algorithm in \cref{sec:patterned-words} to construct a weight-preserving bijection between geodesics in a virtually abelian group and words in a certain language with holonomic generating function.
Lastly, in \cref{sec:blind-multicounter-automata}, we use the algorithm introduced in \cref{sec:patterned-words} to show that the set of all geodesics in a virtually abelian group forms a blind multicounter language.

\section{Notation}\label{sec:notation}

Let $\mathbb{N} = \{0,1,2,\ldots\}$ denote the set of nonnegative integers, including zero, and $\mathbb{N}_+ = \{1,2,3,\ldots\}$ the set of positive integers.

Let $G$ be a group generated as a monoid by a finite weighted set $S$, where each generator $s \in S$ has a positive integer weight $\omega(s) \in \mathbb{N}_+$.
We write $S^*$ for the set of all words in the letters of $S$, and $\overline{\sigma} \in G$ for the group element corresponding to $\sigma \in S^*$.
The \emph{weight} of a word $\sigma = \sigma_1 \sigma_2 \cdots \sigma_k \in S^*$ is then
\[
	\omega(\sigma)
	=
	\omega(\sigma_1) + 
	\omega(\sigma_2) +
	\cdots +
	\omega(\sigma_k).
\]
Moreover, we write $|\sigma|_S = k$ for the \emph{word length} of $\sigma$.
Then the \emph{weighted length} of an element $g \in G$ is given by
\[
	\ell_\omega(g)
	=
	\min\{
		\omega(\sigma)
	\mid
		\overline{\sigma} = g
		\text{ where }
		\sigma \in S^*
	\}.
\]

We say that a word $\sigma \in S^*$ is a \emph{geodesic} if it represents $\overline{\sigma}$ with minimal weight, that is, $\omega(\sigma) = \ell_\omega(\overline{\sigma})$.
Notice then that any sub-word of a geodesic is also a geodesic.
We define the \emph{geodesic growth function} $\gamma_S\colon\mathbb{N} \to \mathbb{N}$ such that $\gamma_S(n)$ counts the geodesic words with weight at most $n$, that is,
\[
	\gamma_{S}(n)
	=
	\#
	\{
		\sigma \in S^*
	\ |\ 
		\omega(\sigma) \leqslant n \text{ and }
		\sigma \text{ is geodesic}
	\}.
\]

We say that $G$ has \emph{polynomial geodesic growth} with respect to $S$ if there are constants $\beta,d\in \mathbb{N}_+$ such that $\gamma_S(n) \leqslant \beta \cdot n^d$ for each $n > 0$; \emph{exponential geodesic growth} with respect $S$ if there is a constant $\alpha \in \mathbb{R}$ with $\alpha>1$ such that $\gamma_S(n) \geqslant \alpha^n$ for each $n \geqslant 0$; and \emph{intermediate geodesic growth} with respect to $S$ if the growth is neither polynomial nor exponential.
The \emph{geodesic growth series} is the power series given by $f_S(z) = \sum_{n=0}^\infty \gamma_S(n) z^n$.

Notice that geodesic growth functions are submultiplicative, that is, for each $n,m\in \mathbb{N}$ we have $\gamma_S(n+m) \leqslant \gamma_S(n) \gamma_S(m)$.
Then, from Fekete's lemma~\cite{fekete1923} we see that the limit $\alpha_S = \lim_{n \to \infty}\sqrt[n]{\gamma_S(n)}$, known as the \emph{growth rate} of $\gamma_S(n)$, is defined.
From this we see that  $G$ has exponential geodesic growth with respect to $S$ if and only if $\alpha_S > 1$, and $\alpha_S = 1$ otherwise.

\subsection{Polyhedral Sets}\label{sec:polyhedral-sets}

Benson~\cite{benson1983} made use of the theory of \emph{polyhedral sets} and their closure properties to show that the standard growth series of virtually abelian groups is rational.
In this paper, we modify these arguments to show a similar result for the geodesic growth series.
Following the convention of Benson~\cite{benson1983}, we define polyhedral sets as follows.

A subset $\mathcal{E} \subseteq \mathbb{Z}^m$ is called an \emph{elementary region} if it can be expressed as
\[
	\left\{
		z \in \mathbb{Z}^m
		\, \middle\vert \,
		a\cdot z = b
	\right\},
	\ 
	\left\{
		z \in \mathbb{Z}^m
		\, \middle\vert \,
		a\cdot z > b
	\right\}
	\text{ or } 
	\left\{
		z \in \mathbb{Z}^m
		\, \middle\vert \,
		a\cdot z \equiv b\ (\bmod\ c)
	\right\}
\]
for some $a \in \mathbb{Z}^m$ and $b,c\in \mathbb{Z}$ with $c > 0$.
A \emph{basic polyhedral set} is a finite intersection of elementary regions;
and a \emph{polyhedral set} is a finite disjoint union of basic polyhedral sets.
It can be seen from this definition that the sets $\emptyset$, $\mathbb{N}^m$ and $\mathbb{Z}^m$ are polyhedral.
We have the following closure properties.

\begin{proposition}[Proposition~13.1~and~Remark~13.2~in~\cite{benson1983}]\label{prop:closure-properties-of-polyhedral-sets}
	The class of polyhedral sets is closed under Cartesian product.
	Moreover, the class of polyhedral sets in $\mathbb{Z}^m$ is closed under finite union, finite intersection and set difference.
\end{proposition}

We say that a map $E \colon \mathbb{Z}^m \to \mathbb{Z}^n$ is an \emph{integer affine transform} if it can be written as $E(v) = vA + b$ where $A \in \mathbb{Z}^{m \times n}$ is a matrix and $b \in \mathbb{Z}^n$ is a vector.
We then have the following additional closure property.

\begin{proposition}[Propositions~13.7~and~13.8~in~\cite{benson1983}]\label{prop:affine-transforms-of-polyhedral-sets}
	Suppose that $\mathcal{P} \subseteq \mathbb{Z}^m$ and $\mathcal{Q} \subseteq \mathbb{Z}^n$ are polyhedral sets, and $E\colon\mathbb{Z}^m\to\mathbb{Z}^n$ is an integer affine transform.
	Then, $E(\mathcal{P})$ and $E^{-1}(\mathcal{Q})$ are both polyhedral sets.
\end{proposition}

\section{Holonomic Power Series}\label{sec:holonomic-functions}

In this section we provide a brief overview to the theory of \emph{holonomic} functions with a focus on their asymptotic properties.
Some authors use the term \emph{D-finite} to refer to the class of single-variable holonomic functions, or as a synonym for holonomic.
For a more complete introduction the reader is directed to \cite{flajolet2009}.
In \cref{sec:polyhedrally-constrained-languages} we study classes of formal languages with holonomic multivariate generating functions, then in \cref{sec:geodesic-growth} we use such a language class to show that the geodesic growth series of a virtually abelian group is holonomic.

To simplify notation, we write $\mathbb{C}[[\mathbf{x}]]$, $\mathbb{C}[\mathbf{x}]$, $\mathbb{C}((\mathbf{x}))$, and $\mathbb{C}(\mathbf{x})$ for the class of formal power series, polynomials, formal Laurent series, and rational functions, respectively, over the variables $\mathbf{x} = (x_1,x_2,\ldots,x_m)$.
Moreover, we write $\partial_{x_i} f(\mathbf{x})$ for the formal partial derivative of $f(\mathbf{x})$ with respect to $x_i$.

A formal power series $f(\mathbf{x}) \in \mathbb{C}[[\mathbf{x}]]$ is \emph{holonomic} if the span of
\[
	X_f=
	\left\{
		\partial_{x_1}^{k_1} \partial_{x_2}^{k_2} \cdots \partial_{x_m}^{k_m}
		f(\mathbf{x})
	\ \middle\vert\ 
		k_1,k_2,\ldots,k_m \in \mathbb{N}
	\right\}
\]
over $\mathbb{C}(\mathbf{x})$ is a finite-dimensional vector space $V_f \subseteq \mathbb{C}((\mathbf{x}))$.
From this definition, a power series of one variable $f(z) \in \mathbb{C}[[z]]$ is holonomic if and only if
\begin{equation}\label{eq:holonomic-single-variable-defn}
	f^{(k+1)}(z)
	+ r_{k}(z) f^{(k)}(z)
	+ \cdots
	+ r_1(z) f'(z)
	+ r_0(z) f(z)
	=
	0
\end{equation}
for some $k$ where $r_0(z),r_1(z),\ldots,r_{k}(z) \in \mathbb{C}(z)$ are rational functions.

\begin{lemma}\label{lemma:holonomic-power-series-poles}
	If $f(z) \in \mathbb{C}[[z]]$ is a holonomic power series and complex analytic in a neighbourhood of the identity, then $f(z)$ can be analytically extended to all but finitely many points in $\mathbb{C}$, that is, $f(z)$ can have only finitely many poles.
\end{lemma}

\begin{proof}
	Let $f(z) \in \mathbb{C}[[z]]$ be a holonomic power series that is complex analytic in the open neighbourhood $N \subseteq \mathbb{C}$ of the identity.
	Since $f(z)$ is holonomic, it must satisfy a differential equation as in (\ref{eq:holonomic-single-variable-defn}) with rational coefficients $r_i(z) \in \mathbb{C}(z)$.
	Let $X$ be the finite set of poles of the rational functions $r_i(z)$.
	Then, it is sufficient to show that $f(z)$ can be analytically extended to each point in $\mathbb{C} \setminus X$.
	
	Let $p \in \mathbb{C} \setminus X$ be chosen arbitrarily.
	Let $R \subseteq \mathbb{C} \setminus X$ be an open and simply connected region for which $p \in R$ and the intersection $R \cap N$ is nonempty.
	Then, each coefficient $r_i(z)$ of (\ref{eq:holonomic-single-variable-defn}) is complex analytic in $R$.
	From the existence and uniqueness theorem of \cite[Theorem~2.2]{wasow1965} we see that $f(z)$ can be analytically extended from $R \cap N$ to $p \in R$.
\end{proof}

The class of holonomic functions enjoys many interesting closure properties, however, in this paper we only require the following.

\begin{lemma}[Proposition~2.3~in~\cite{lipshitz1989}]\label{lemma:holonomic-closure-properties}
	The class of holonomic functions in $\mathbb{C}[[\mathbf{x}]]$ is closed under addition and multiplication.
	If $f(\mathbf{x}) \in \mathbb{C}[[\mathbf{x}]]$ is holonomic where $\mathbf{x} = (x_1,x_2,\ldots,x_m)$, and $a_1(\mathbf{y}), a_2(\mathbf{y}), \ldots, a_m(\mathbf{y}) \in \mathbb{C}[[\mathbf{y}]]$ are algebraic where $\mathbf{y} = (y_1,y_2,\ldots,y_n)$, then $g(\mathbf{y}) = f(a_1(\mathbf{y}),a_2(\mathbf{y}),\ldots,a_m(\mathbf{y}))$ is holonomic if it is defined.
	Moreover, each algebraic function is holonomic.
\end{lemma}

\subsection{Rational power series}\label{sec:holonomic-functions/rational}

Let $f(z) = \sum_{n=0}^\infty c_n z^n$ be a rational power series, then the growth of the sequence $(c_n)_{n=0}^\infty$ can be computed with the use of the following lemma.

\begin{lemma}[Theorem~IV.9~in~\cite{flajolet2009}]\label{lemma:rational-polynomial-exponential}
	Suppose that $f(z) = \sum_{n=0}^\infty c_n z^n$ is rational with poles at $\alpha_1,\alpha_2,\ldots,\alpha_k \in \mathbb{C}$.
	Then there are polynomials $p_i(n) \in \mathbb{C}[n]$ such that for each sufficiently large $n$ we have $c_n = \sum_{j = 0}^k p_j(n) \alpha_j^{-n}$.
\end{lemma}

If $(c_n)_{n\in \mathbb{N}}$ is an integer sequence with $\lim_{n\to\infty} \sqrt[n]{|c_n|} = 1$, then the power series $f(z) = \sum_{n=0}^\infty c_n z^n$ is complex analytic in the open unit disc.
For such a power series, we may apply the P\'olya-Carlson theorem, given in \cref{lemma:polya-carlson}.

\begin{lemma}[Carlson~{\cite[p.~3]{carlson1921}}]\label{lemma:polya-carlson}
	If $f(z)$ is a power series with integer coefficients that is complex analytic in the open unit disc, then $f(z)$ is either rational or has the unit circle as its natural boundary.
\end{lemma}

From this we obtain the following characterisation of groups with holonomic geodesic growth series.

\begin{corollary}\label{cor:dichotomy-of-characterisation}
	Let $G$ be a group with a finite (weighted monoid) generating set $S$.
	If the geodesic growth series $f_S(z) = \sum_{n=0}^\infty \gamma_S(n) z^n$ is holonomic, then $G$ either has exponential geodesic growth with respect to $S$, or  $f_S(z)$ is rational and $G$ has polynomial geodesic growth with respect to $S$.
\end{corollary}

\begin{proof}
As we mentioned in \cref{sec:notation}, from Fekete's Lemma~\cite{fekete1923} we know that either the geodesic growth function $\gamma_S(n)$ has exponential growth, or the geodesic growth rate $\alpha_S = \lim_{n \to \infty} \sqrt[n]{\gamma_S(n)} = 1$.
In the latter case we apply \cref{lemma:holonomic-power-series-poles,lemma:polya-carlson} to show that  $f_S(z)$ is rational, then from \cref{lemma:rational-polynomial-exponential} we see that there is a polynomial upper bound on $\gamma_S(n)$.
\end{proof}

\section{Polyhedrally Constrained Languages}\label{sec:polyhedrally-constrained-languages}

Massazza~\cite{massazza1993} studied the class of \emph{linearly constrained languages}, and showed that their multivariate generating functions are holonomic.
Informally, a \emph{linearly constrained language} is the intersection of an unambiguous context-free language and the set of words whose Parikh images satisfy a set of linear constraints.
In this section we generalise this result by instead requiring that the Parikh images belong to a polyhedral set, we refer to such languages as \emph{polyhedrally constrained}.

In \cref{sec:polyhedrally-constrained-languages/formal-language} we recall the definition of multivariate generating functions for formal languages, and in \cref{sec:polyhedrally-constrained-languages/constrained-language} we define the families of linearly and polyhedrally constrained languages.
We then show that polyhedrally constrained languages have holonomic multivariate generating functions in \cref{prop:polyhedrally-constrained-is-holonomic} by decomposing such languages into a union of finitely many disjoint linearly constrained languages.

To introduce the family of linearly constrained languages we first define what it means for a subset of $\mathbb{Z}^n$ to be an \emph{$n$-constraint}.
Modifying the notation of Massazza~\cite{massazza1993}, a subset of $\mathbb{Z}^{n}$ is an \emph{$n$-atom} if it can be expressed as
\[
	\{
		v \in \mathbb{Z}^{n}
	\mid
		a \cdot v = b
	\}
	\quad
	\text{or}
	\quad 
	\{
		v \in \mathbb{Z}^{n}
	\mid
		a \cdot v > b
	\}
\]
where $a \in \mathbb{Z}^{n}$ and $b \in \mathbb{Z}$.
An \emph{$n$-constraint} is a Boolean expression of $n$-atoms, that is, a finite expression of $n$-atoms using intersection, union, and complement with respect to $\mathbb{Z}^{n}$.
For example,
\[
	\{
		(x,y) \in \mathbb{Z}^{2}
	\mid
		\text{either }
		x = 1 \text{ and } y > 10, \text{ or }
		x \neq 1\text{ and }2x-3y>4
	\}
\]
is a $2$-constraint as it can be written as the Boolean expression
\begin{multline*}
		\{ v \in \mathbb{Z}^{2} \mid (1,0)\cdot v = 1 \} \cap
		\{ v \in \mathbb{Z}^{2} \mid (0,1)\cdot v > 10 \}\\
	\cup
		\left(
		\mathbb{Z}^{2}\setminus
		\{ v \in \mathbb{Z}^{2} \mid (1,0)\cdot v = 1 \}
		\right)
		\cap
		\{ v \in \mathbb{Z}^{2} \mid (2,-3)\cdot v > 4 \}.
\end{multline*}

Notice that our definition of $n$-atoms is similar to that of elementary regions, except we do not allow modular arithmetic.
From the closure properties in \cref{prop:closure-properties-of-polyhedral-sets} we see that $n$-constraints form a subclass of the polyhedral subsets of $\mathbb{Z}^n$.
It can be verified by the reader, for each $n \geqslant 1$, that
\[
	\{
		(x,0,0,\ldots,0) \in \mathbb{Z}^n
	\mid
		x \equiv 0\ (\bmod\ 2)
	\}
\]
is a polyhedral set but not an $n$-constraint, and thus the class of $n$-constraints form a proper subclass of the polyhedral sets.

\subsection{Formal language generating functions}\label{sec:polyhedrally-constrained-languages/formal-language}

Let $\Sigma = \{\sigma_1, \sigma_2, \ldots, \sigma_m \}$ be an ordered alphabet, then the \emph{Parikh map} for words in $\Sigma^*$ is the monoid homomorphism $\Phi \colon \Sigma^* \to \mathbb{N}^{|\Sigma|}$ defined such that $\Phi(\sigma_i) = e_i \in \mathbb{N}^{|\Sigma|}$ for each $\sigma_i \in \Sigma$ where $e_i$ is the $i$-th standard basis element.
Then for each word $w \in \Sigma^*$ we have
\[
	\Phi(w)
	=
	(
		|w|_{\sigma_1},
		|w|_{\sigma_2},
		\ldots,
		|w|_{\sigma_m}
	)
\]
where each $|w|_{\sigma_i}$ counts the number of occurrences of the letter $\sigma_i$ in $w$.

The \emph{multivariate generating function} of a language $L \subseteq \Sigma^*$ is given by
\[
	f(x_1,x_2,\ldots,x_m)
	=
	\sum_{i_1,i_2,\ldots,i_m \in \mathbb{N}}
		c(i_1,i_2,\ldots,i_m)
		x_1^{i_1} x_2^{i_2} \cdots x_m^{i_m}
\]
where each coefficient $c(i_1,i_2,\ldots,i_m)$ is given as
\[
	c(i_1,i_2,\ldots,i_m)
	=
	\#
	\left\{
		w \in L
	\mid
		\Phi(w) = (i_1,i_2,\ldots,i_m)
	\right\}.
\]

We may now define constrained languages as follows.

\subsection{Constrained languages}\label{sec:polyhedrally-constrained-languages/constrained-language}

Let $U \subseteq \Sigma^*$ be an unambiguous context-free language, and let $\mathcal{C} \subseteq \mathbb{Z}^{|\Sigma|}$ be a subset of $\mathbb{Z}^{|\Sigma|}$, then we say that
\[
	L(U,\mathcal{C}) = \{ w \in U \mid \Phi(w) \in \mathcal{C} \}
\]
is a \emph{constrained language}.
If the set $\mathcal{C}$ is a $|\Sigma|$-constraint then we say that $L(U,\mathcal{C})$ is a \emph{linearly constrained language}.
Moreover, if $\mathcal{C}$ is a polyhedral set, as in \cref{sec:polyhedral-sets}, then we say that $L(U,\mathcal{C})$ is a \emph{polyhedrally constrained language}.

\begin{remark}\label{rmk:assume-constraint-is-nonnegative}
	Notice that $\Phi(w) \in \mathbb{N}^{|\Sigma|}$ for each word $w \in L(U,\mathcal{C}) \subseteq \Sigma^*$, and thus $L(U,\mathcal{C}) = L(U, \mathcal{C} \cap \mathbb{N}^{|\Sigma|})$.
\end{remark}

Massazza~\cite{massazza1993} showed the following result for linearly constrained languages.

\begin{proposition}[Theorem~2~in~\cite{massazza1993}]\label{prop:lcl-is-holonomic}
	The multivariate generating function of a linearly constrained language is holonomic.
\end{proposition}

We extend this to the class of polyhedrally constrained languages as follows.

\begin{proposition}\label{prop:polyhedrally-constrained-is-holonomic}
	The multivariate generating function of a polyhedrally constrained language is holonomic.
\end{proposition}

\begin{proof}

Let $L(U,\mathcal{P}) \in \Sigma^*$ be a polyhedrally constrained language.
From the definition of polyhedral sets in \cref{sec:polyhedral-sets}, we may decompose $\mathcal{P}$ into a finite union, $\mathcal{P} = \bigcup_{i=1}^L \mathcal{B}_i$, of disjoint basic polyhedral sets.
Moreover, each such basic polyhedral set $\mathcal{B}_i \subseteq \mathbb{Z}^{|\Sigma|}$ can be written as a finite intersection of the form
\begin{multline*}
	\mathcal{B}_i
	=
	\bigcap_{j = 1}^{K_{i,1}}
		\{
			v \in \mathbb{Z}^{|\Sigma|}
		\mid
			\alpha_{i,j} \cdot v = \beta_{i,j}
		\}
	\cap
	\bigcap_{j = 1}^{K_{i,2}}
		\{
			v \in \mathbb{Z}^{|\Sigma|}
		\mid
			\xi_{i,j} \cdot v > \lambda_{i,j}
		\}
	\\\cap
	\bigcap_{j = 1}^{K_{i,3}}
		\{
			v \in \mathbb{Z}^{|\Sigma|}
		\mid
			\zeta_{i,j} \cdot v \equiv \eta_{i,j}\ (\bmod\ \theta_{i,j})
		\}
\end{multline*}
where each $\alpha_{i,j},\xi_{i,j},\zeta_{i,j} \in \mathbb{Z}^{|\Sigma|}$, each $\beta_{i,j},\lambda_{i,j}, \eta_{i,j} \in \mathbb{Z}$, and each $\theta_{i,j} \in \mathbb{N}_+$.

From the definition of constrained language we see that $L(U,\mathcal{P})$ is the union of disjoint constrained languages $L(U,\mathcal{B}_i)$.
We see that if each $L(U,\mathcal{B}_i)$ has a multivariate generating function of $f_i(x_1,x_2,\ldots,x_{|\Sigma|})$, then the multivariate generating function for $L(U,\mathcal{P})$ is given by
\[
	f(x_1,x_2,\ldots,x_{|\Sigma|})
	=
	\sum_{i=1}^L f_i(x_1,x_2,\ldots,x_{|\Sigma|}).
\]

For each basic polyhedral set $\mathcal{B}_i$, we introduce a $|\Sigma|$-constraint
\[
	\mathcal{C}_i
	=
	\bigcap_{j = 1}^{K_{i,1}}
		\{
			v \in \mathbb{Z}^{|\Sigma|}
		\mid
			\alpha_{i,j} \cdot v = \beta_{i,j}
		\}
	\cap
	\bigcap_{j = 1}^{K_{i,2}}
		\{
			v \in \mathbb{Z}^{|\Sigma|}
		\mid
			\xi_{i,j} \cdot v > \lambda_{i,j}
		\},
\]
and a monoid homomorphism $\varphi_i \colon \Sigma^* \to \prod_{j=1}^{K_{i,3}} (\mathbb{Z}/\theta_{i,j}\mathbb{Z})$ such that
\[
	\varphi_i(w)
	=
	(
		\zeta_{i,1} \cdot \Phi(w),\,
		\zeta_{i,2} \cdot \Phi(w),\,
		\ldots,\,
		\zeta_{i,K_{i,3}} \cdot \Phi(w)
	);
\]
moreover, we write $R_i \in \Sigma^*$ for the inverse image
\[
	R_i =
	\varphi_i^{-1}(\{ (\eta_{i,1},\eta_{i,2},\ldots,\eta_{i,K_{i,3}}) \}).
\]

Each language $R_i \in \Sigma^*$ is expressed as the inverse image of a subset of a finite monoid.
From \cite[Theorem~1]{rabin1959} we see that each $R_i$ is a regular language, in particular, for each $R_i$ we may construct a finite-state automaton with states given by the set $\prod_{j=1}^{K_{i,3}} (\mathbb{Z}/\theta_{i,j}\mathbb{Z})$, initial state given by $(0,\ldots,0)$, an accepting state of $(\eta_{i,1},\eta_{i,2},\ldots,\eta_{i,K_{i,3}})$, and a transition $v \to^\sigma v'$ for each state $v$ and letter $\sigma \in \Sigma$ where $v' = v + \varphi_i(\sigma)$.
Moreover, since the class of unambiguous context-free grammars is closed under intersection with regular languages, we see that each $L(U \cap R_i,\mathcal{C}_i) = L(U,\mathcal{B}_i)$ is linearly constrained.
Then, from \cref{prop:lcl-is-holonomic}, we see that each $f_i(x_1,x_2,\ldots,x_{|\Sigma|})$ is holonomic.

From \cref{lemma:holonomic-closure-properties}, holonomic functions are closed under addition, and thus the multivariate generating function of $L(U, \mathcal{P})$ is holonomic.
\end{proof}

\section{Patterned Words in Virtually Abelian Groups}\label{sec:patterned-words}

In the remainder of this paper $G$ will denote a virtually abelian group that is generated as a monoid by some finite weighted generating set $S$.
It is well known that $G$ contains a finite-index normal subgroup that is isomorphic to $\mathbb{Z}^n$ for some $n$.
Without loss of generality, we assume that $\mathbb{Z}^n \triangleleft G$ with $d = [G:\mathbb{Z}^n]$.
Fix a set of coset representatives $T = \{t_1=1,t_2,\ldots,t_d\}$ for $\mathbb{Z}^n$ in $G$, then we write elements of $G$ in the normal form $g = z \cdot t$ where $z \in \mathbb{Z}^n$ and $t \in T$.

\begin{definition}\label{defn:phi-rho}
	Let $\psi\colon G \to \mathbb{Z}^n$ and $\rho\colon G \to T$ be the maps defined such that the normal form for $g \in G$ is given by $\psi(g) \cdot \rho(g)$.
\end{definition}

Benson~\cite{benson1983} showed that virtually abelian groups have rational standard growth series by demonstrating that each group element has at least one geodesic representative that can be expressed as a \emph{patterned word}, where the set of such patterned words is then studied using the theory of polyhedral sets.
In this section we modify these arguments to study the set of all geodesic words in $S^*$, in particular, we describe \cref{algo:word-shuffling} which converts words in $S^*$ to \emph{patterned words} which represent the same group element with the same weight.
In \cref{sec:patterned-words/properties} we compute the weight and group element of patterned words, and describe the patterned words which correspond to geodesics.

We begin by defining two finite sets of words $Y,P \subseteq S^*$ as follows.

\begin{definition}\label{defn:sets-Y-P}
From the generating set $S$ and the normal subgroup $\mathbb{Z}^n \triangleleft G$ with finite index $d = [G : \mathbb{Z}^n]$, we define the sets
\begin{align*}
	Y
	&=
		\{
			\sigma \in S^*
		\mid
			1 \leqslant |\sigma|_S \leqslant d
			\ \,\mathrm{and}\,\ 
			\overline{\sigma} \in \mathbb{Z}^n
		\}
	\ \ \text{and}
	\\
	P
	&=
		\{
			\sigma \in S^*
		\mid
			1 \leqslant |\sigma|_S \leqslant d - 1
			\ \,\mathrm{and}\,\ 
			\overline{\sigma} \notin \mathbb{Z}^n
		\},
\end{align*}
and we fix a labelling $\{y_1,y_2,\ldots,y_m\} = Y$ where $m = |Y|$.
\end{definition}

We define the sets $Y$ and $P$ as above so that we have the technical property given in \cref{lemma:factoring-words}.
We will find this property useful in the proof of \cref{lemma:map-delta} which is then used to construct \cref{algo:word-shuffling}.

\begin{lemma}\label{lemma:factoring-words}
	Suppose that $w \in S^*$ with $1 \leqslant |w|_S \leqslant d$ and $w \notin P$.
	Then, there is a factoring $w = \alpha \beta \delta$ with $\alpha \in P \cup \{ \varepsilon \}$, $\beta \in Y$ and $\delta \in S^*$.
	In particular, there is a unique choice of such a factoring for which $(|\alpha|_S, |\beta|_S) \in \mathbb{N}^2$ is minimal with respect to the lexicographic ordering on $\mathbb{N}^2$.
\end{lemma}

\begin{proof}

Let $w = w_1 w_2 \cdots w_k$ with $1 \leqslant k \leqslant d$ and $w \notin P$.

Notice that if we have at least one such factorisation, then there is a unique choice of such a factoring where $(|\alpha|_S, |\beta|_S) \in \mathbb{N}^2$ is minimal with respect to the lexicographic ordering on $\mathbb{N}^2$.
Thus, all that remains to be shown is that at least one such factoring $w = \alpha\beta\delta$ exists.

If $|w|_S < d$, then we have such a factorisation given by $\beta = w$, and $\alpha = \delta = \varepsilon$.
Thus, in the remainder of this proof we consider the case where $|w|_S = d$.

If $|w|_S = d$, then from the pigeonhole principle on the $d$ cosets, we see that there must be a nontrivial factor $b = w_i w_{i+1} \cdots w_j$ for which $\overline{b} \in \mathbb{Z}^n$.
Let $I \geqslant 1$ be the smallest value for which there is a $J \geqslant I$ with $\overline{w_{I} w_{I+1} \cdots w_J} \in \mathbb{Z}^n$, then let $\alpha = w_1 w_2 \cdots w_{I-1}$ and $\beta = w_{I} w_{I+1} \cdots w_J$.
From our choice of indices $I$ and $J$, we see that $\beta \in Y$, and either $\alpha = \varepsilon$ or $\overline{\alpha} \notin \mathbb{Z}^n$.
Moreover, we see that $|\alpha|_S = I-1 \leqslant d - 1$ and thus $\alpha \in P\cup \{\varepsilon\}$.
\end{proof}

Notice that $S \subseteq Y \cup P$, and thus $Y \cup P$ generates the group $G$.
We will see that for each word $\sigma \in S^*$, there is a word $w \in Y^*(PY^*)^k$, with $0 \leqslant k \leqslant d$, such that $w$ represents the same group element as $\sigma$ with the same weight.
We formalise this by defining \emph{patterns} and \emph{patterned words} as follows.

\begin{definition}[Patterned words]\label{defn:patterned-words}
	Let $\pi = \pi_1 \pi_2 \cdots \pi_k \in P^*$ be a word in the letters of $P$ with length $k = |\pi|_P \leqslant d$ for which each proper prefix belongs to a distinct coset, that is,
	\begin{equation}\label{eq:pattern-cosets}
		1 = \rho(\overline{\varepsilon}),\ 
		\rho(\overline{\pi_1}),\ 
		\rho(\overline{\pi_1\pi_2}),\ 
		\ldots,\ 
		\rho(\overline{\pi_1\pi_2\cdots\pi_{k-1}})
	\end{equation}
	are pairwise distinct; and let $v \in \mathbb{N}^{(k+1)m}$ be a vector where $m = |Y|$.
	Then we say that $\pi$ is a \emph{pattern} and that $(v,\pi)$ is a \emph{patterned word}.
	We then write
	\begin{multline*}
	v^\pi
	=
	\Big(
		y_1^{v_1}
		y_2^{v_2}
		\cdots
		y_m^{v_m}
	\Big)
	\pi_1
	\Big(
		y_1^{v_{m+1}}
		y_2^{v_{m+2}}
		\cdots
		y_m^{v_{2m}}
	\Big)
	\pi_2
	\\
	\cdots
	\pi_k
	\Big(
		y_1^{v_{k \cdot m+1}}
		y_2^{v_{k \cdot m+2}}
		\cdots
		y_m^{v_{(k+1)\cdot m}}
	\Big).
	\end{multline*}
	Notice that $\rho(\overline{\pi})$ is not included in (\ref{eq:pattern-cosets}).
	If $\rho(\overline{\pi})$ is also distinct from each coset representative in (\ref{eq:pattern-cosets}), then we say that $\pi$ is a \emph{strong pattern} and that $(v,\pi)$ is a \emph{strongly patterned word}.
\end{definition}

To simplify notation in later sections, we introduce the following sets.

\begin{definition}\label{defn:pattern-set}
	We write $\text{\sc Patt} \subseteq P^*$ for the set of all patterns, and we write $\text{\sc StrPatt} \subseteq \text{\sc Patt}$ for the set of all strong patterns.
	Notice that $\text{\sc Patt}$ and $\text{\sc StrPatt}$ are finite, in particular, $|\text{\sc Patt}| \leqslant |P|^{d+1}$.
\end{definition}

To simplify notation in \cref{sec:patterned-words/word-shuffling-algorithm}, we extend this as follows.

\begin{definition}[Extended Patterned Words]\label{defn:extended-special-form}
	If $(v,\pi)$ is a (strongly) patterned word, and $\sigma \in S^*$, then $((v,\pi),\sigma)$ is an \emph{extended (strongly) patterned word}.
\end{definition}

In \cref{algo:word-shuffling}, for each word $\sigma \in S^*$, we construct a finite sequence of extended patterned words that begins with $((\mathbf{0},\varepsilon),\sigma)$ and ends with an extended patterned word of the form $((v,\pi),\varepsilon)$.
Moreover, this sequence has the property that $v^\pi$ and $\sigma$ represent the same group element with the same weight.
To simplify notation, we define the following equivalence relation.

\begin{definition}\label{defn:equiv-relation}
	We define the equivalence relation $\simeq$ on $S^*$ such that, for each $w,\sigma \in S^*$, we have $w \simeq \sigma$ if and only if both $\overline{w} = \overline{\sigma}$ and $\omega(w) = \omega(\sigma)$.
\end{definition}

Notice that if we have a patterned word $(v,\pi)$ with $v^\pi \simeq \sigma$, then $\sigma$ is a geodesic if and only if the word $v^\pi$ is a geodesic.

\subsection{Word shuffling}\label{sec:patterned-words/word-shuffling-algorithm}

In this subsection, we construct \cref{algo:word-shuffling} which `shuffles' words of the form $\sigma \in S^*$ into patterned words $(v,\pi)$.
In particular, for each word $\sigma$, we compute a finite sequence of extended patterned words
\begin{multline}\label{eq:patterned-sequence}
	((\mathbf{0},\varepsilon),\sigma)
	=
	((u^{(1)},\tau^{(1)}),\sigma^{(1)}),
	((u^{(2)},\tau^{(2)}),\sigma^{(2)}),
	\\\ldots,
	((u^{(q)},\tau^{(q)}),\sigma^{(q)})
	=
	((v,\pi),\varepsilon),
\end{multline}
such that
\[
	(u^{(i)})^{\tau^{(i)}} \sigma^{(i)}
	\simeq
	(u^{(i+1)})^{\tau^{(i+1)}}\sigma^{(i+1)}
	\quad
	\text{and}
	\quad
	|\sigma^{(i)}|_S > |\sigma^{(i+1)}|_S
\]
for each $i$.
Notice that $v^\pi \simeq \sigma$, and $q \leqslant |\sigma|_S+1$ where $q$ is the length of the sequence in (\ref{eq:patterned-sequence}).
From (\ref{eq:patterned-sequence}), we define $\mathrm{Shuffle}(\sigma) = (v,\pi)$ where the patterned word $(v,\pi)$ has the property that $v^\pi \simeq \sigma$.

The idea of \cref{algo:word-shuffling} is to compute each $((u^{(i+1)},\tau^{(i+1)}),\sigma^{(i+1)})$ from its previous extended patterned word $((u^{(i)},\tau^{(i)}),\sigma^{(i)})$ by replacing a bounded-length prefix of $\sigma^{(i)}$ with a strictly shorter word, adding at most a unit vector to $u^{(i)}$, and adding at most one letter to $\tau^{(i)}$.
In order to describe our algorithm, we introduce the following additional notation.

Recall that $d = [G : \mathbb{Z}^n]$ is the index of the $\mathbb{Z}^n$ normal subgroup of $G$.
For each word $\sigma \in S^*$, we fix a bounded-length prefix as follows.

\begin{definition}[Prefixes]\label{rmk:short-prefix}
	We write $\mathrm{Prefix} \colon S^* \to S^*$ for the function which computes the prefix of a word of length at most $d$, that is, $\mathrm{Prefix}(\sigma) = \sigma_1 \sigma_2 \cdots \sigma_q$ where $q = \min(d,|\sigma|_S)$.
	Notice that if $w = \mathrm{Prefix}(\sigma)$ with $|w|_S < d$, then $\sigma = w$.
\end{definition}

In sequence (\ref{eq:patterned-sequence}), each word $\sigma^{(i+1)}$ is obtained from $\sigma^{(i)}$ by replacing the prefix $w^{(i)} = \mathrm{Prefix}(\sigma^{(i)})$ with a strictly shorter word $w^{(i)\prime}$.
We write these prefix replacements using the following notation.

\begin{definition}[Prefix Replacements]
Let $\sigma \in S^*$ be a word which factors as $\sigma = w\zeta$ where $w,\zeta \in S^*$, then for each word $w' \in S^*$ we write $(w \mapsto w') \cdot \sigma = w'\zeta$ which we call a \emph{prefix replacement}.
We write a sequence of replacements as
\[
	(w_n \mapsto w'_n)
	\cdots
	(w_2 \mapsto w'_2)
	(w_1 \mapsto w'_1)
	\cdot \sigma
\]
where replacements are composed right-to-left.
Notice that if $\sigma' = (w \mapsto w')\cdot \sigma$, then $\omega(\sigma') = \omega(\sigma) - \omega(w) + \omega(w')$ where $\omega \colon S^* \to \mathbb{N}$ is the weight function.
\end{definition}

To understand how prefix replacements are composed, consider the following.

\begin{example}\label{ex:prefix-replacement-sequence}
We have the sequence of replacements
\begin{equation}\label{eq:prefix-replacement-example}
	(c \mapsto dc)
	(ba \mapsto cb)
	(\varepsilon \mapsto b)
	\cdot
	az
	=
	dcbz.
\end{equation}
Notice that if $(w \mapsto w')\cdot \sigma$ is defined, then we have $\sigma = (w' \mapsto w)(w \mapsto w')\cdot \sigma$, that is, each prefix replacement has an inverse.
For example, from the sequence of prefix replacements given in (\ref{eq:prefix-replacement-example}), we see that
\[
	az =
	(b \mapsto \varepsilon)
	(cb \mapsto ba)
	(dc \mapsto c)
	\cdot
	dcbz.
\]
Thus, we may compute the inverse of a sequence of prefix replacements.
\end{example}

For each pattern $\pi$, we write $\mathcal{N}_\pi$ for the set of all vectors $v$ for which $(v,\pi)$ is a patterned word, as defined in \cref{defn:patterned-words}.
We introduce the following notation to simplify the description of our algorithm.

\begin{definition}\label{defn:standard-basis-elements}
For each pattern $\pi = \pi_1 \pi_2 \cdots \pi_k \in P^*$, we write $\mathcal{Z}_\pi$ and $\mathcal{N}_\pi$ for the sets $\mathbb{Z}^{(k+1)m}$ and $\mathbb{N}^{(k+1)m}$, respectively, where $m = |Y|$.
Moreover, for each $i \in \{ 1,2,\ldots,\dim(\mathcal{Z}_\pi) \}$ we write $e_{\pi,i}$ for the $i$-th standard basis element of $\mathcal{Z}_\pi$ and $e_{\pi,\varnothing} = \mathbf{0} \in \mathcal{Z}_\pi$ for the zero vector of $\mathcal{Z}_\pi$.
\end{definition}

When computing (\ref{eq:patterned-sequence}), it may be the case that $|\tau^{(i)}|_P \neq |\tau^{(i+1)}|_P$ and thus the vectors $u^{(i)}$ and $u^{(i+1)}$ lie in different spaces $\mathcal{N}_{\tau^{(i)}}$ and $\mathcal{N}_{\tau^{(i+1)}}$, respectively.
We define the following map to convert between these spaces.

\begin{definition}\label{defn:projection}
For each pair of patterns $\pi,\tau \in P^*$, let $t = \dim(\mathcal{Z}_\tau)$ and $p = \dim(\mathcal{Z}_\pi)$, then we define the map $\mathrm{Proj}_{\pi,\tau}\colon \mathcal{Z}_\pi \to \mathcal{Z}_\tau$ such that
\[
	\mathrm{Proj}_{\pi,\tau}(u_1,u_2,\ldots,u_p)
	=
	(u_1,u_2,\ldots,u_p,0,0,\ldots,0)
\]
if $t > p$, and
\[
	\mathrm{Proj}_{\pi,\tau}(u_1,u_2,\ldots,u_p)
	=
	(u_1,u_2,\ldots,u_t)
\]
otherwise.
Notice that if $\dim(\mathcal{Z}_\tau) < \dim(\mathcal{Z}_\pi)$, then $\mathrm{Proj}_{\pi,\tau}$ is a projection; otherwise, $\dim(\mathcal{Z}_\tau) \geqslant \dim(\mathcal{Z}_\pi)$ and $\mathrm{Proj}_{\pi,\tau}$ is an embedding.
\end{definition}

In order to construct \cref{algo:word-shuffling}, we need to define a map which explicitly describes how to construct the sequence of extended patterned words in~(\ref{eq:patterned-sequence}).
We construct such a map in the following lemma.

\begin{lemma}\label{lemma:map-delta}
	We may construct a map
	\[
		\Delta \colon
		\text{\sc StrPatt} \times W_1
		\to
		(\mathbb{N}_+ \cup \varnothing) \times \text{\sc Patt} \times W_2,
	\]
	where
	\[
		W_1 = \{w \in S^* \mid 1 \leqslant |w|_S \leqslant d\}
		\quad\text{and}\quad
		W_2 = \{w \in S^* \mid |w|_S < d\}
	\]
	with the following properties.
	Let $((u,\tau),\sigma)$ be an extended strongly patterned word, and let $\Delta(\tau,w) = (x,\tau',w')$ with $w = \mathrm{Prefix}(\sigma)$.
	We may then \emph{apply} $\Delta$ to obtain an extended patterned word $((u',\tau'),\sigma')$ where $u' = \mathrm{Proj}_{\tau,\tau'}(u)+e_{\tau',x}$ and $\sigma' = (w\mapsto w') \cdot \sigma$.
	This will be denoted as
	\[
		((u,\tau),\sigma)
		\xrightarrow{\Delta}
		((u',\tau'),\sigma').
	\]
	For each extended strongly patterned word $((u,\tau),\sigma)$,
	\begin{enumerate}
		\item\label{lemma:map-delta:prop1}
			$|\tau|_P \leqslant |\tau'|_P$ and thus $\mathrm{Proj}_{\tau,\tau'}\colon\mathcal{N}_\tau \to \mathcal{N}_{\tau'}$ is an embedding;
		\item\label{lemma:map-delta:prop2}
			$u^\tau \sigma \simeq (u')^{\tau'}\sigma'$;
		\item\label{lemma:map-delta:prop3}
			$|\sigma|_S > |\sigma'|_S$ and $\omega(w) > \omega(w')$; and
		\item\label{lemma:map-delta:prop4}
			either $|\sigma'|_S = 0$, or $((u',\tau'),\sigma')$ is an extended strongly patterned word.
	\end{enumerate}
	Notice that property~\ref{lemma:map-delta:prop4} implies that either $((u',\tau'),\sigma')$ is equivalent to the patterned word $(u',\tau')$, or we may apply $\Delta$ again.
	From property~\ref{lemma:map-delta:prop3}, we see that after finitely many applications of the map $\Delta$, we have a patterned word.
\end{lemma}

\begin{proof}

Let $\tau = \tau_1 \tau_2 \cdots \tau_k \in P^*$ be a strong pattern, that is, $\tau$ is a pattern for which the coset representatives
\[
	\rho(\overline{\varepsilon}),\,
	\rho(\overline{\tau_1}),\,
	\rho(\overline{\tau_1 \tau_2}),\,
	\rho(\overline{\tau_1 \tau_2 \tau_3}),\,
	\ldots,\,
	\rho(\overline{\tau})
\]
are pairwise distinct.
Then, from the pigeonhole principle on the $d$ cosets of $\mathbb{Z}^n$ in $G$, we see that $|\tau|_P = k < d$.

Let $w \in S^*$ be a word with length $1 \leqslant |w|_S \leqslant d$.
We separate the remainder of this proof into the cases where $w \in P$ and $w \notin P$ as follows.

Suppose that $w \in P$, then we have a length $k+1$ pattern $\tau' = \tau w$,
moreover, from the definition of words in $P$, we see that $|w|_S < d$, and from \cref{rmk:short-prefix} we have $w = \sigma$.
We then define $\Delta(\tau,w) = (\varnothing, \tau', \varepsilon)$.
For each extended strongly patterned word $((u,\tau),w)$ with $u \in \mathbb{N}^p$, we then obtain an extended patterned word $((u',\tau'),\varepsilon)$ where $u' = \mathrm{Proj}_{\tau,\tau'}(u) = (u_1,u_2,\ldots,u_p,0,0,\ldots,0)$.
Notice that we have $(u')^{\tau'} = u^\tau w$.
This completes our proof for the case that $w \in P$.

In the remainder of this proof, we suppose that $w \notin P$.
From \cref{lemma:factoring-words}, we factor $w$ uniquely as $w = \alpha\beta\delta$ where $\alpha \in P \cup \{\varepsilon\}$, $\beta \in Y$ and $(|\alpha|_S,|\beta|_S)$ is minimal with respect to the lexicographic order on $\mathbb{N}^2$.
From the labelling $Y = \{y_1,y_2,\ldots,y_m\}$, we see that there must be an index $b$ such that $\beta = y_b$.

Let $((u,\tau),\sigma)$ be an extended strongly patterned word with $w = \mathrm{Prefix}(\sigma)$ and $u = (u_0, u_1, \ldots,u_k)$ where each $u_{a} = (u_{a,1}, u_{a,2},\ldots,u_{a,m}) \in \mathbb{N}^m$. 
Then if we factor $\sigma$ as $\sigma = w\zeta$, we see that
\begin{multline*}
	u^\tau \sigma
	=
	\Big(
		y_1^{u_{0,1}}
		y_2^{u_{0,2}}
		\cdots
		y_m^{u_{0,m}}
	\Big)
	\tau_1
	\Big(
		y_1^{u_{1,1}}
		y_2^{u_{1,2}}
		\cdots
		y_m^{u_{1,m}}
	\Big)
	\tau_2
	\\
	\cdots
	\tau_k
	\Big(
		y_1^{u_{k,1}}
		y_2^{u_{k,2}}
		\cdots
		y_m^{u_{k,m}}
	\Big)
	\alpha y_b \delta \zeta.
\end{multline*}
If there is an index $a$ with $0 \leqslant a \leqslant k$ such that $\rho(\overline{\tau_1 \tau_2 \cdots \tau_a}) = \rho(\overline{\pi \alpha})$, then the choice of such an index $a$ must be unique, and we see that
\[
	\overline{
		\tau_{a+1}
		\Big(
			y_1^{u_{a+1,1}}
			y_2^{u_{a+1,2}}
			\cdots
			y_m^{u_{a+1,m}}
		\Big)
		\tau_{a+2}
		\cdots
		\tau_k
		\Big(
			y_1^{u_{k,1}}
			y_2^{u_{k,2}}
			\cdots
			y_m^{u_{k,m}}
		\Big)
		\alpha
	} \in \mathbb{Z}^n
\]
commutes with $\overline{y_b} \in \mathbb{Z}^n$, that is,
\[
	{(u_0,\ldots,u_{a-1},u_a+e_b,u_{a+1},\ldots,u_k)}^\tau
	\alpha \delta \zeta
	\simeq
	u^\tau \alpha y_b \delta \zeta
	=
	u^\tau \sigma
\]
where $e_b \in \mathbb{N}^m$ is the $b$-th standard basis element.
In this case we define the map $\Delta(\tau,w) = (a\cdot m+b, \, \tau,\, \alpha\delta)$ and our proof is complete.
Otherwise, we see that the coset representatives
\[
	\rho(\overline{\varepsilon}),\,
	\rho(\overline{\tau_1}),\,
	\rho(\overline{\tau_1\tau_2}),\,
	\rho(\overline{\tau_1\tau_2\tau_3}),\,
	\ldots,\,
	\rho(\overline{\tau}),\,
	\rho(\overline{\tau\alpha})
\]
are pairwise distinct and $\alpha \neq \varepsilon$, that is, $\alpha \in P$.
Then we see that the length $k+1$ word $\tau' = \tau \alpha \in P^*$ is a strong pattern, and that we have
\[
	(u_0,u_2,\ldots,u_k,e_b)^{\tau'} \delta \zeta = u^\tau \alpha y_b \delta\zeta = u^\tau\sigma
\]
where $e_b \in \mathbb{N}^m$ is the $b$-th standard basis vector.
Then after defining the map $\Delta(\tau,w) = (a\cdot k+b, \tau', \delta)$ our proof is complete.
\end{proof}

We are now ready to define our algorithm as follows.

\begin{algorithm}[Word Shuffling]\label{algo:word-shuffling}
	Let $\Delta$ be the map in \cref{lemma:map-delta}.
	For each word $\sigma \in S^*$, there is a finite sequence of extended patterned words
	\begin{multline}\label{algo:word-shuffling/sequence}
		((\mathbf{0},\varepsilon),\sigma)
		=
		((u^{(1)},\tau^{(1)}),\sigma^{(1)})
		\xrightarrow{\Delta}
		((u^{(2)},\tau^{(2)}),\sigma^{(2)})
		\\\cdots
		\xrightarrow{\Delta}
		((u^{(q)},\tau^{(q)}),\sigma^{(q)})
		=
		((v,\pi),\varepsilon).
	\end{multline}
	From this sequence we define $\mathrm{Shuffle}(\sigma) = (v,\pi)$.
	Notice from property~\ref{lemma:map-delta:prop3} in \cref{lemma:map-delta} that we have
	\[
		|\sigma|_S = |\sigma^{(1)}|_S >
		|\sigma^{(2)}|_S >
		|\sigma^{(3)}|_S >
		\cdots >
		|\sigma^{(q)}|_S
		= 0
	\]
	and thus $q \leqslant |\sigma|_S + 1$.
From property~\ref{lemma:map-delta:prop2} in \cref{lemma:map-delta}, we see that $v^\pi \simeq \sigma$.
\end{algorithm}

In the remainder of this section, we compute the group elements and weights of patterned words, and determine which patterned word represents geodesics.

\subsection{Geodesic patterned words}\label{sec:patterned-words/properties}

From \cref{algo:word-shuffling}, for each word $\sigma \in S^*$ we have a well-defined patterned word $(v,\pi) = \mathrm{Shuffle}(\sigma)$ such that $v^\pi \simeq \sigma$, that is, $v^\pi$ represents the same group element as $\sigma$ with the same weight.
In particular, we see that $\sigma$ is a geodesic if and only if $v^\pi$ is a geodesic.

In this section, we modify an argument of Benson~\cite{benson1983} and show that the group element and weight of any word $v^\pi$ can be computed with the use of integer affine transforms, and that we may verify that $v^\pi$ is a geodesic by checking if the vector $v$ belongs to a polyhedral set $\mathcal{G}_\pi$.

\begin{lemma}\label{lemma:patterned-word-maps}
	For each pattern $\pi$, there are integer affine transformations $\Psi_{\pi} \colon \mathcal{Z}_\pi \to \mathbb{Z}^n$ and $\Omega_{\pi} \colon \mathcal{Z}_\pi \to \mathbb{Z}$ such that for each patterned word $(v,\pi)$, we have $\overline{v^\pi} = \Psi_\pi(v) \cdot \rho(\overline{\pi})$ and $\omega(v^\pi) = \Omega_\pi(v)$.
\end{lemma}

\begin{proof}

Recall that in \cref{defn:sets-Y-P} we fixed a labelling $Y = \{y_1, y_2, \ldots, y_m\}$ where $m  = |Y|$.
Define the matrix $Z \in \mathbb{Z}^{m \times n}$ such that $e_i Z = \overline{y_i}$ for each standard basis vector $e_i \in \mathbb{Z}^m$.
Then, we see that $vZ = \overline{y_1^{v_1}y_2^{v_2} \cdots y_m^{v_m}}$ for each $v \in \mathbb{N}^m$.
For each $p \in P$ we see that $\overline{p} x \overline{p}^{-1} \in \mathbb{Z}^n$ for each $x \in \mathbb{Z}^n \triangleleft G$; thus we define matrices $R_p \in \mathbb{Z}^{n \times n}$ such that $x R_p = \overline{p} x \overline{p}^{-1}$ for each $x \in \mathbb{Z}^n$.

To compute the element $\overline{v^\pi}$ we first rewrite $v^\pi$ as
\begin{multline*}
	\Big(
		y_1^{v_{1}}
		y_2^{v_{2}}
		\cdots
		y_m^{v_{m}}
	\Big)
	\cdot
	\pi_1
	\Big(
		y_1^{v_{m+1}}
		y_2^{v_{m+2}}
		\cdots
		y_m^{v_{2m}}
	\Big)
	\pi_1^{-1}
	\\
	(\pi_1 \pi_2)
	\Big(
		y_1^{v_{2m+1}}
		y_2^{v_{2m+2}}
		\cdots
		y_m^{v_{3m}}
	\Big)
	(\pi_1 \pi_2)^{-1}
	\\
	\cdots
	\pi
	\Big(
		y_1^{v_{km+1}}
		y_2^{v_{km+2}}
		\cdots
		y_m^{v_{(k+1)m}}
	\Big)
	\pi^{-1}
	\cdot
	\pi.
\end{multline*}
Then we see that $\rho(\overline{v^\pi}) = \rho(\overline{\pi})$ and $\psi(\overline{v^\pi}) = \Psi_{\pi}(v)$ where
\begin{multline*}
	\Psi_{\pi}(v)
	=
	(v_1,v_2,\ldots,v_m) Z +
	(v_{m+1},v_{m+2},\ldots,v_{2m}) Z R_{\pi_1} +\\
	\cdots +
	(v_{mk+1},v_{mk+2},\ldots,v_{(k+1)m}) Z R_{\pi_k} \cdots R_{\pi_2} R_{\pi_1}
	+ \psi(\pi).
\end{multline*}
Considering the word $v^\pi$ we see that $\omega(v^\pi) = \Omega_{\pi}(v)$ where
\[
	\Omega_{\pi}(v)
	=
	\omega(\pi) +
	\sum_{j=0}^k \sum_{i=1}^m
		v_{jm+i} \cdot \omega(y_i).
\]
The maps $\Psi_{\pi} \colon \mathcal{Z}_\pi \to \mathbb{Z}^n$ and $\Omega_{\pi} \colon \mathcal{Z}_\pi \to \mathbb{Z}$ are integer affine transforms.
\end{proof}

From the integer affine transformations defined in \cref{lemma:patterned-word-maps} and the closure properties of polyhedral sets we have the following result.

\begin{lemma}\label{lemma:geodesics-in-special-form}
	For each pattern $\pi$, there is a polyhedral set $\mathcal{G}_{\pi} \subseteq \mathcal{N}_\pi$ such that $v \in \mathcal{G}_{\pi}$ if and only if $(v,\pi)$ is a patterned word where $v^\pi$ is a geodesic.
\end{lemma}

\begin{proof}
From \cref{algo:word-shuffling} we see that the word $v^\pi$ is a geodesic if and only if there is no patterned word $(u,\tau)$ with $\overline{u^\tau} = \overline{v^\pi}$ and $\omega(u^\tau) < \omega(v^\pi)$.
For each pattern $\pi$, let $E_{\pi} \colon \mathcal{Z}_\pi \to \mathbb{Z}^{n+1}$ be the integer affine transformation defined as $E_{\pi} (v) = (\Psi_{\pi}(v),\Omega_{\pi}(v))$,
and let $\mathcal{R} \subseteq \mathbb{Z}^{2(n+1)}$ be the polyhedral set
\[
	\mathcal{R}
	=
	\left\{
		(\nu,\mu) \in \mathbb{Z}^{n+1} \times \mathbb{Z}^{n+1}
	\,\middle\vert\,
	\begin{aligned}
		\nu_1 = \mu_1,\,
		\nu_2 = \mu_2,\,
		\ldots,\,
		\nu_n = \mu_n\\
		\text{and }
		\nu_{n+1} > \mu_{n+1}
	\end{aligned}
	\right\}.
\]
Then, we see that $v^\pi$ is geodesic if and only if there is no patterned word $(u,\tau)$ with $\rho(\overline{\tau}) = \rho(\overline{\pi})$ and
$
	\big(E_{\pi}(v),E_{\tau}(u)\big)
	\in
	\mathcal{R}
$; or equivalently, $v^\pi$ is a geodesic if and only if the intersection
\[
	\Big(
		E_{\pi}(\{v\})
		\times
		E_{\tau}\!
			\left(\mathcal{N}_{\tau}\right)
	\Big)
	\cap
	\mathcal{R}
\]
is empty for each pattern $\tau$ with $\rho(\overline{\tau})=\rho(\overline{\pi})$.

Let $f \colon \mathbb{Z}^{n+1}\times \mathbb{Z}^{n+1} \to \mathbb{Z}^{n+1}$ be the projection onto the first $\mathbb{Z}^{n+1}$ factor, that is, $f(\nu,\mu)=\nu$ for each $(\nu,\mu) \in \mathbb{Z}^{n+1} \times \mathbb{Z}^{n+1}$.
Let
\[
	\mathcal{D}_{\pi,\tau}
	=
	\mathcal{N}_{\pi}
	\cap
	\left[
		\left(E_{\pi}\right)^{-1}
		f\left(
		\Big(
			E_{\pi}\!
			\left(\mathcal{N}_{\pi}\right)
			\times
			E_{\tau}\!
			\left(\mathcal{N}_{\tau}\right)
			\Big)
			\cap
			\mathcal{R}
		\right)
	\right]
.\]
Then, we see that $v^\pi$ is a geodesic if and only if $v \notin \mathcal{D}_{\pi,\tau}$ for each pattern $\tau$ with $\rho(\overline{\tau})=\rho(\overline{\pi})$.
Then, $v^\pi$ is a geodesic if and only if $v \in \mathcal{G}_{\pi}$ where
\[
	\mathcal{G}_{\pi}
	=
	\mathcal{N}_{\pi}
	\setminus
	\bigcup
	\Big\{
		\mathcal{D}_{\pi,\tau}
	\ \Big\vert\ 
		\tau
		\text{ is a pattern with }
		\rho(\overline{\tau})=\rho(\overline{\pi})
	\Big\}
\]
where we see that the above union is finite as there can be only finitely many patterns.
Moreover, each set $\mathcal{G}_{\pi} \subseteq \mathcal{N}_\pi$ is polyhedral from the closure properties in \cref{prop:affine-transforms-of-polyhedral-sets,prop:closure-properties-of-polyhedral-sets}.
\end{proof}

\section{Geodesic Growth}\label{sec:geodesic-growth}

In this section we provide a characterisation of the geodesic growth of virtually abelian groups, in particular, we show that the geodesic growth of a virtually abelian group with respect to any finite (weighted monoid) generating set is either polynomial with rational geodesic growth series, or exponential with holonomic geodesic growth series.
This result is provided in \cref{thm:geodesic-growth}.

In \cref{lemma:path-to-special-form}, we construct a weight-preserving bijection from the words in $S^*$ to a subset of paths in a weighted graph $\Gamma$.
Then, in \cref{thm:geodesic-growth} we construct a weight-preserving bijection from the set of such paths which correspond to geodesics, and the set of words in a finite number of polyhedrally constrained languages.
Using the theory developed in \cref{sec:holonomic-functions}, we then prove our result.
We begin by constructing the graph $\Gamma$ as follows.

\begin{definition}\label{defn:graph-gamma}
	Let $\Delta$ be the map constructed in \cref{lemma:map-delta}.
	Let $\Gamma$ be the finite weighted directed edge-labelled graph defined as follows.
	For each pattern $\tau$ and each word $w \in S^*$ with $|w|_S \leqslant d = [G:\mathbb{Z}^n]$, the graph $\Gamma$ has a vertex $[\tau,w] \in \mathrm{V}(\Gamma)$.
	Suppose $\tau \in \text{\sc StrPatt}$, $|w|_S \geqslant 1$ and that $\Delta(\tau,w) = (x,\tau',w')$;
	if $|w|_S = d$, then for each word $\xi \in S^*$ with $|w'\xi|_S \leqslant d$, the graph $\Gamma$ has a labelled edge $[\tau,w] \xrightarrow{x} [\tau',w'\xi]$;
	otherwise, $1 \leqslant |w|_S < d$ and the graph $\Gamma$ has a labelled edge $[\tau,w] \xrightarrow{x} [\tau',w']$.
	Moreover, each such edge has weight $\omega(w)-\omega(w') > 0$.
\end{definition}

We are interested in paths of the following form.

\begin{definition}\label{defn:path-sets}
	For each pattern $\pi$, we write $\text{\sc Path}_\pi$ for the set of paths
	\[
	\text{\sc Path}_\pi
	=
	\left\{
		p\colon
		[\varepsilon,w]
		\to^*
		[\pi, \varepsilon]
	\mid
		w \in S^* \text{ with }|w|_S \leqslant d
	\right\}.
	\]
	We write $\text{\sc Path}$ for the union of all such sets, that is,
	$
		\text{\sc Path}
		=
		\bigcup_\pi \text{\sc Path}_\pi
	$.
\end{definition}

We count the instances of each edge label as follows.

\begin{definition}
	Let $\alpha \colon \text{\sc Path} \to \bigcup\{ \mathcal{N}_\pi \mid \pi \in \text{\sc Patt} \}$ map paths $p \in \text{\sc Path}_\pi$ to vectors in $\mathcal{N}_\pi$ such that the $i$-th component of $\alpha(p)$ counts the number of edges of $p$ that are labelled with $i$, that is, if
	$
		p\colon
		\nu_1 \xrightarrow{x_1}
		\nu_2 \xrightarrow{x_2}
		\cdots \xrightarrow{x_k}
		[\pi,\varepsilon] \in \text{\sc Path}_\pi,
	$
	then we have $\alpha(p) = v \in \mathcal{N}_\pi$ where each $v_i = \#\{j \mid x_j = i\}$.
\end{definition}

In the following lemma, we construct a weight-preserving bijection that maps from the set of words $S^*$ to the set of paths $\text{\sc Path}$.

\begin{lemma}\label{lemma:path-to-special-form}
	We may construct a weight-preserving bijection $S^* \to \text{\sc Path}$, which we denote as $\sigma \mapsto p_\sigma$, with the following properties.
	For each word $\sigma \in S^*$ where $\mathrm{Shuffle}(\sigma) = (v,\pi)$, we have $p_\sigma \in \text{\sc Path}_\pi$.
	For each path $p \in \text{\sc Path}_\pi$, there is a unique word $\sigma \in S^*$ such that $p = p_\sigma$ and $v^\pi \simeq \sigma$ where $v = \alpha(p) \in \mathcal{N}_\pi$.
\end{lemma}

\begin{proof}
Let $\sigma \in S^*$, then from \cref{algo:word-shuffling} there is a finite sequence
\begin{multline}\label{eq:delta-sequence}
	((\mathbf{0},\varepsilon),\sigma)
	=
	((u^{(1)},\tau^{(1)}),\sigma^{(1)})
	\xrightarrow{\Delta}
	((u^{(2)},\tau^{(2)}),\sigma^{(2)})
	\\\cdots
	\xrightarrow{\Delta}
	((u^{(q)},\tau^{(q)}),\sigma^{(q)})
	=
	((v,\pi),\varepsilon).
\end{multline}
Let  $w^{(j)} = \mathrm{Prefix}(\sigma^{(j)})$ and $\Delta(\tau^{(j)},w^{(j)}) = (x^{(j+1)},\tau^{(j+1)},w^{(j)\prime})$.

From \cref{lemma:map-delta} we see that each $\sigma^{(j+1)} = (w^{(j)}\mapsto w^{(j)\prime})\cdot \sigma^{(j)}$.
Then,
\begin{equation}\label{eq:path-to-word}
	\sigma
	=
	(w^{(1)\prime}\mapsto w^{(1)})
	(w^{(2)\prime}\mapsto w^{(2)})
	\cdots
	(w^{(q-1)\prime}\mapsto w^{(q-1)})
	\cdot
	\varepsilon
\end{equation}
and thus
\begin{multline}\label{eq:total-weight}
	\omega(\sigma)
	=
	(\omega(w^{(1)}) - \omega(w^{(1)\prime}))
	+ (\omega(w^{(2)}) - \omega(w^{(2)\prime}))\\
	+ \cdots
	+ (\omega(w^{(q-1)}) - \omega(w^{(q-1)\prime})).
\end{multline}
Moreover, from the properties of $\Delta$ given in \cref{lemma:map-delta}, and the definition of the graph $\Gamma$ given in \cref{defn:graph-gamma}, we see that
\begin{multline}\label{eq:path-p}
	p_\sigma \colon
	[\varepsilon,w]
	=
		[\tau^{(1)}, w^{(1)}]
	\xrightarrow{x^{(2)}}
		[\tau^{(2)}, w^{(2)}]
	\xrightarrow{x^{(3)}}\\
	\cdots
	\xrightarrow{x^{(q-1)}}
		[\tau^{(q-1)}, w^{(q-1)}]
	\xrightarrow{x^{(q)}}
		[\tau^{(q)}, w^{(q)}]
	= [\pi, \varepsilon]
\end{multline}
is a path in $\text{\sc Path}_\pi$.
Notice that the weight of the path $p_\sigma$ is the same as the weight of $\sigma$ in (\ref{eq:total-weight}) and thus the map $\sigma \mapsto p_\sigma$ is weight preserving.
It remains to be shown that the map $\sigma \mapsto p_\sigma$ is a bijection.

Suppose that we are given a path $p_\sigma$ as in (\ref{eq:path-p}).
Then, we may recover the words $w^{(i)\prime}$ as $\Delta(\tau^{(j)},w^{(j)}) = (x^{(j+1)},\tau^{(j+1)},w^{(j)\prime})$.
Hence, we may recover the word $\sigma$ using equation (\ref{eq:path-to-word}).
Thus, we see that the map $\sigma \mapsto p_\sigma$ is one-to-one.
It remains to be shown that the map $\sigma \mapsto p_\sigma$ is onto, that is, for each $p \in \text{\sc Path}$, there is a word $\sigma$ such that $p = p_\sigma$.

Let $p \in \text{\sc Path}_\pi$ be a path written as 
\begin{multline*}
	p \colon
	[\varepsilon,w]
	=
	[\tau^{(1)}, w^{(1)}]
	\xrightarrow{x^{(2)}}
	[\tau^{(2)}, w^{(2)}]
	\xrightarrow{x^{(3)}}\\
	\cdots
	\xrightarrow{x^{(q-1)}}
	[\tau^{(q-1)}, w^{(q-1)}]
	\xrightarrow{x^{(q)}}
	[\tau^{(q)}, w^{(q)}]
	= [\pi, \varepsilon].
\end{multline*}
Let $\Delta(\tau^{(j)},w^{(j)}) = (x^{(j+1)},\tau^{(j+1)},w^{(j)\prime})$.
We define the words $\sigma^{(j)}$ such that
\[
	\sigma^{(j)}
	=
	(w^{(j)\prime}\mapsto w^{(j)})
	\cdot
	\sigma^{(j+1)}
\]
and $\sigma^{(q)} = \varepsilon$.
We define the vectors $u^{(j)} \in \mathcal{N}_{\tau^{(j)}}$ such that
\[
	u^{(j+1)}
	=
	\mathrm{Proj}_{\tau^{(j)},\tau^{(j+1)}}(u^{(j)})
	+ e_{\tau^{(j+1)},x^{(j+1)}}
\]
and $u^{(1)} = \mathbf{0} \in \mathcal{N}_{\tau^{(1)}}$.
From the property~\ref{lemma:map-delta:prop1} in \cref{lemma:map-delta} we see that each $|\tau^{(j)}|_P \leqslant |\tau^{(j+1)}|_P$, and thus from \cref{defn:projection} we see that
$u^{(q)} = \alpha(p)$.
Let $\sigma = \sigma^{(1)}$ and $v = u^{(q)}$, then we see that
\begin{multline*}
	((\mathbf{0},\varepsilon),\sigma)
	=
	((u^{(1)},\tau^{(1)}),\sigma^{(1)})
	\xrightarrow{\Delta}
	((u^{(2)},\tau^{(2)}),\sigma^{(2)})
	\\\cdots
	\xrightarrow{\Delta}
	((u^{(q)},\tau^{(q)}),\sigma^{(q)})
	=
	((v,\pi),\varepsilon).
\end{multline*}
From this, we see that $\mathrm{Shuffle}(\sigma) = (v,\pi)$ and that $p = p_\sigma$ with $\sigma \simeq v^\pi$ where $v = \alpha(p) \in \mathcal{N}_\pi$.
Moreover, we see that the map $\sigma \mapsto p_\sigma$ is onto.
\end{proof}

We may now prove our first main theorem as follows.

\begin{theorem}\label{thm:geodesic-growth}
	Let $G$ be a virtually abelian group with a finite (weighted monoid) generating set $S$.
	Then either the geodesic growth with respect to $S$ is polynomial with rational geodesic growth series; or the geodesic growth with respect to $S$ is exponential with holonomic geodesic growth series.
\end{theorem}

\begin{proof}

From \cref{lemma:path-to-special-form}, we may compute the geodesic growth function as
\begin{equation}\label{eq:geod-growth-sum}
	\gamma_S(k)
	=
	\sum_{\pi \in \text{\sc Patt}}
	\#
	\{
		p \in \text{\sc Path}_\pi
	\mid
		\omega(p) \leqslant k
		\text{ and }\alpha(p) \in \mathcal{G}_\pi
	\}
\end{equation}
where $\omega(p)$ is the weight of $p$, and $\mathcal{G}_\pi$ is the polyhedral set in \cref{lemma:geodesics-in-special-form}.
Notice that (\ref{eq:geod-growth-sum}) is a finite sum as we have only finitely many patterns.

Let $\Sigma$ be the weighted finite alphabet which contains the edges of $\Gamma$, that is, for each edge $e\colon \nu_1 \xrightarrow{x} \nu_2$ in $\Gamma$, there is a letter $(\nu_1,x,\nu_2) \in \Sigma$ with weight $\omega(e)$.
Then, for each pattern $\pi$, we have a weight-preserving bijection from the paths in $\text{\sc Path}_\pi$ to the words in a language $L_\pi \subseteq \Sigma^*$, in particular, the language $L_\pi$ contains all words of the form
\[
	([\varepsilon,w],x_1,\nu_1)
	(\nu_1,x_2,\nu_2)
	(\nu_2,x_3,\nu_3)
	\cdots
	(\nu_k,x_{k+1},[\pi,\varepsilon]) \in \Sigma^*.
\]
Notice that each $L_\pi$ is a regular language.

We write $\Phi(\nu_1,x,\nu_2) \in \mathbb{N}^{|\Sigma|}$ to denote the Parikh vector corresponding to the letter $(\nu_1,x,\nu_2) \in \Sigma$.
For each pattern $\pi$, we define an integer affine transform $E_\pi \colon \mathbb{Z}^{|\Sigma|} \to \mathcal{Z}_\pi$ such that $E_\pi(\Phi(\nu_1,x,\nu_2)) = e_{\pi,x}$ is the $x$-th standard basis element for each $x \in \{1,2,\ldots,\dim(\mathcal{Z}_\pi)\}$, and $E_\pi(\Phi(\nu_1,x,\nu_2)) = \mathbf{0}$ otherwise.
Let $w \in L_\pi$ be the word corresponding to the path $p \in \text{\sc Path}_\pi$, then we see that $\alpha(p) = E_\pi(\Phi(w))$.
From \cref{lemma:path-to-special-form}, we see that the path $p$ corresponds to a geodesic if and only if $\Phi(w) \in E^{-1}_\pi (\mathcal{G}_\pi)$.

For each pattern $\pi$, we define the constrained language $L^\mathrm{geod}_\pi \subseteq L_\pi$ as
\[
	L_\pi^\mathrm{geod}
	=
	\{
		w \in L_\pi
	\mid
		\Phi(w) \in E^{-1}(\mathcal{G}_\pi)
	\}.
\]
Notice that there is a weight-preserving bijection between $L^\mathrm{geod}_\pi$ and the set of geodesics $\sigma \in S^*$ with $p_\sigma \in \text{\sc Path}_\pi$.
From \cref{prop:affine-transforms-of-polyhedral-sets} we see that $E^{-1}(\mathcal{G}_\pi)$ is a polyhedral set and thus each $L^\mathrm{geod}_\pi$ is a polyhedrally constrained language, as studied in \cref{sec:polyhedrally-constrained-languages}.
Then, from \cref{prop:polyhedrally-constrained-is-holonomic} we see that the multivariate generating function $f_\pi(x_1,x_2,\ldots,x_{|\Sigma|})$ of each $L^\mathrm{geod}_\pi$ is holonomic.

Let $a_{\pi,i}\in \mathbb{N}_+$ be the weight of the letter that corresponds to the variable $x_{i}$ in the generating function $f_\pi(x_1,x_2,\ldots,x_{|\Sigma|})$.
Let $h_\pi(z) \in \mathbb{C}[[z]]$ be defined as
\[
	h_\pi(z) =
	f_\pi(z^{a_{\pi,1}}, z^{a_{\pi,2}}, \ldots, z^{a_{\pi,|\Sigma|}}).
\]
Then we see that the coefficient of $z^k$ in $h_\pi(z)$ counts the geodesics $\sigma \in S^*$ for which $p_\sigma \in \text{\sc Path}_\pi$ and $\omega(\sigma) = k$.
Let $g(z) \in \mathbb{C}[[z]]$ be defined as
\[
	g(z) = \frac{1}{1-z} \cdot \sum_{\pi \in \text{\sc Patt}} h_\pi(z),
\]
Then we see that the coefficient of $z^k$ in $g(z)$ is given by $\gamma_S(k)$, that is, $g(z)$ is the geodesic growth series $g(z) = \sum_{k = 0}^\infty \gamma_S(k) z^k$.
Moreover, from the closure properties in \cref{lemma:holonomic-closure-properties} we see that the function $g(z)$ is holonomic.

Our result then follows from \cref{cor:dichotomy-of-characterisation}.
\end{proof}

\section{Language of Geodesics}\label{sec:blind-multicounter-automata}

In the previous section we found a weight-preserving bijection between the geodesics in a virtually abelian group and some subset of paths in a finite graph $\Gamma$, and showed that the language of all such paths is \emph{polyhedrally constrained} as defined in \cref{sec:polyhedrally-constrained-languages}.
However, this bijection cannot be used to immediately generalise this result to a characterisation of the language of geodesics.
Instead, in this section we show that the language of geodesics for a virtually abelian group belongs to the family of \emph{blind multicounter languages}.

Informally, a \emph{blind $k$-counter automaton}, as studied by Greibach~\cite{greibach1978}, is a nondeterministic finite-state acceptor with a one-way input tape and $k$ integer counters;
where such a machine is allowed to increment and decrement its counters by fixed amounts only during transitions.
Moreover, the transitions of such a machine are not allowed to depend on the state of its integer counters.
A computation of such a machine begins with zero on all its counters, and accepts when it is in an accepting state with all input consumed and zero on all its counters.
A language $L$ is called \emph{blind multicounter} if there is a blind $k$-counter automaton which accepts precisely $L$.

Blind multicounter automata satisfy the language hierarchy in \cref{fig:language-hierarchy}, in particular, we may construct a language for each section of \cref{fig:language-hierarchy} as follows.
We see that the class of finite-state automata is equivalent to the class of blind $0$-counter automata, and that each blind $k$-counter language is also a blind $(k+1)$-counter language.
We see that blind 1-counter language is a subclass of context-free.
Moreover, from \cite[Theorem~1]{greibach1978} it can be seen that the class of blind multicounter languages is a subclass of context-sensitive languages.
It is straightforward to see that the word problem for $F_2 \times F_2$ is context-sensitive (see~\cite{shapiro1994}).
It is a classic result by Muller and Schupp~\cite{muller1983} that the word problem for a group is context-free if and only if the group is virtually free.
Moreover, it was shown in \cite{elder2008} that the word problem for a group is blind $k$-counter if and only if the group is virtually $\mathbb{Z}^m$ for some $m \leqslant k$.
From these characterisations we see that the word problem for the free group $F_2$ is context-free but not blind multicounter; for each $k \geqslant 2$, the word problem for $\mathbb{Z}^k$ is blind $k$-counter but not context-free; the word problem for $\mathbb{Z}^{k+1}$ is blind $(k+1)$-counter but not blind $k$-counter; and that the word problem for $F_2 \times F_2$ is context-sensitive and neither context-free nor blind multicounter.
From the proof of Theorem~5 in \cite{greibach1978}, we see that
\[
	L_k
	=
	\{
		a_1^{n_1}
		a_2^{n_2}
		\cdots
		a_k^{n_k}
		b_k^{n_k}
		\cdots
		b_2^{n_2}
		b_1^{n_1}
	\mid
		n_1,n_2,\ldots,n_k \in \mathbb{N}
	\}
\]
is context-free and blind $k$-counter, but not blind $(k-1)$-counter.

\begin{figure}[!ht]
	\centering
	
	\begin{tikzpicture}[xscale=3.5,yscale=2]
	
	\coordinate (zero) at (0,0);
	
	\node [ellipse,above=0pt of zero,draw] (regular) {regular};
	\coordinate (regular-top) at ($(regular.north)$);
	
	%%%%%%%%%%%%%%%%%%%%%%%%%%%%%%%%%%%%%%%%%%%%%%%%%%%%%%
	
	\node [above=0.25em of regular-top] (1-counter-text) {blind 1-counter};
	
	\draw [inner sep=0pt,outer sep=0pt,draw] let
		\p1=($(1-counter-text.north)$),
		\p2=($(1-counter-text.east)$),
		\n1={\x2/sin(180 - 2*atan2(\y1,\x2))}
	in
		(0,\n1) circle (\n1);
	\path let
		\p1=($(1-counter-text.north)$),
		\p2=($(1-counter-text.east)$),
		\n1={\x2/sin(180 - 2*atan2(\y1,\x2))}
	in coordinate
		(1-counter-top) at (0,2*\n1);
	
	%%%%%%%%%%%%%%%%%%%%%%%%%%%%%%%%%%%%%%%%%%%%%%%%%%%%%%
	
	\node [below=0pt of zero] (context-free-text) {context-free};
	
	\draw [inner sep=0pt,outer sep=0pt,draw,thick] let
		\p1=($(1-counter-top)$),
		\p2=($(context-free-text.east)$),
		\p3=($(context-free-text.south)$),
		\p0=($(\x2,\y1-\y3+0.25em)$),
		\n1={\x0/sin(180 - 2*atan2(\x0,\y0))}
	in
		($(0,\y1-\n1+0.25em)$) ellipse (1.5*\n1 and \n1);
	\path let
		\p1=($(1-counter-top)$),
		\p2=($(context-free-text.east)$),
		\p3=($(context-free-text.south)$),
		\p0=($(\x2,\y1-\y3+0.25em)$),
		\n1={\x0/sin(180 - 2*atan2(\x0,\y0))}
	in coordinate
		(context-free-bottom) at (0,\y1-2*\n1+0.25em);
	
	%%%%%%%%%%%%%%%%%%%%%%%%%%%%%%%%%%%%%%%%%%%%%%%%%%%%%%
	
	\node [above=0.5em of 1-counter-top] (2-counter-text) {blind 2-counter};
	
	\draw [inner sep=0pt,outer sep=0pt,draw] let
		\p1=($(2-counter-text.north)$),
		\p2=($(2-counter-text.east)$),
		\n1={\x2/sin(180-2*atan2(\x2,\y1))}
	in
		(0,\n1) circle (\n1);
	\path let
		\p1=($(2-counter-text.north)$),
		\p2=($(2-counter-text.east)$),
		\n1={\x2/sin(180-2*atan2(\x2,\y1))}
	in coordinate
		(2-counter-top-real) at (0,2*\n1);
	
	%%%%%%%%%%%%%%%%%%%%%%%%%%%%%%%%%%%%%%%%%%%%%%%%%%%%%%
	
	\node [above=0em of 2-counter-top-real] (2-counter-top) {$\makeatletter
		\vbox{
			\baselineskip4\p@\lineskiplimit\z@
			\kern-\p@
			\hbox{.}\hbox{.}\hbox{.}
		}
		\makeatother$};
	
	%%%%%%%%%%%%%%%%%%%%%%%%%%%%%%%%%%%%%%%%%%%%%%%%%%%%%%
	
	\node [above=0em of 2-counter-top] (3-counter-text) {blind $k$-counter};
	
	\draw [inner sep=0pt,outer sep=0pt,draw] let
		\p1=($(3-counter-text.north)$),
		\p2=($(3-counter-text.east)$),
		\n1={\x2/sin(180-2*atan2(\x2,\y1))}
	in
		(0,\n1) circle (\n1);
	\path let
		\p1=($(3-counter-text.north)$),
		\p2=($(3-counter-text.east)$),
		\n1={\x2/sin(180-2*atan2(\x2,\y1))}
	in coordinate
		(3-counter-top) at (0,2*\n1);
	
	%%%%%%%%%%%%%%%%%%%%%%%%%%%%%%%%%%%%%%%%%%%%%%%%%%%%%%
	
	\node [above=0em of 3-counter-top] (dots) {$\makeatletter
		\vbox{
			\baselineskip4\p@\lineskiplimit\z@
			\kern-\p@
			\hbox{.}\hbox{.}\hbox{.}
		}
		\makeatother$};
	
	%%%%%%%%%%%%%%%%%%%%%%%%%%%%%%%%%%%%%%%%%%%%%%%%%%%%%%
	
	\node [above=0.25em of dots] (k-counter-text) {blind multicounter};
	
	\draw [inner sep=0pt,outer sep=0pt,draw] let
		\p1=($(k-counter-text.north)$),
		\p2=($(k-counter-text.east)$),
		\n1={\x2/sin(180-2*atan2(\x2,\y1))}
	in
		(0,\n1) circle (\n1);
	\path let
		\p1=($(k-counter-text.north)$),
		\p2=($(k-counter-text.east)$),
		\n1={\x2/sin(180-2*atan2(\x2,\y1))}
	in coordinate
		(k-counter-top) at (0,2*\n1);
	
	%%%%%%%%%%%%%%%%%%%%%%%%%%%%%%%%%%%%%%%%%%%%%%%%%%%%%%
	
	\node [below=0.25em of context-free-bottom] (context-sensitive-text) {context-sensitive};
	
	\draw [inner sep=0pt,outer sep=0pt,draw] let
		\p1=($(k-counter-top)$),
		\p2=($(context-sensitive-text.east)$),
		\p3=($(context-sensitive-text.south)$),
		\p0=($(\x2,\y1-\y3+0.25em)$),
		\n1={\x0/sin(180-2*atan2(\x0,\y0))}
	in
		($(0,\y1-\n1+0.25em)$) circle (\n1);
	
	%%%%%%%%%%%%%%%%%%%%%%%%%%%%%%%%%%%%%%%%%%%%%%%%%%%%%%
	
	\end{tikzpicture}
	
	\caption{Hierarchy of blind multicounter language.}\label{fig:language-hierarchy}
\end{figure}
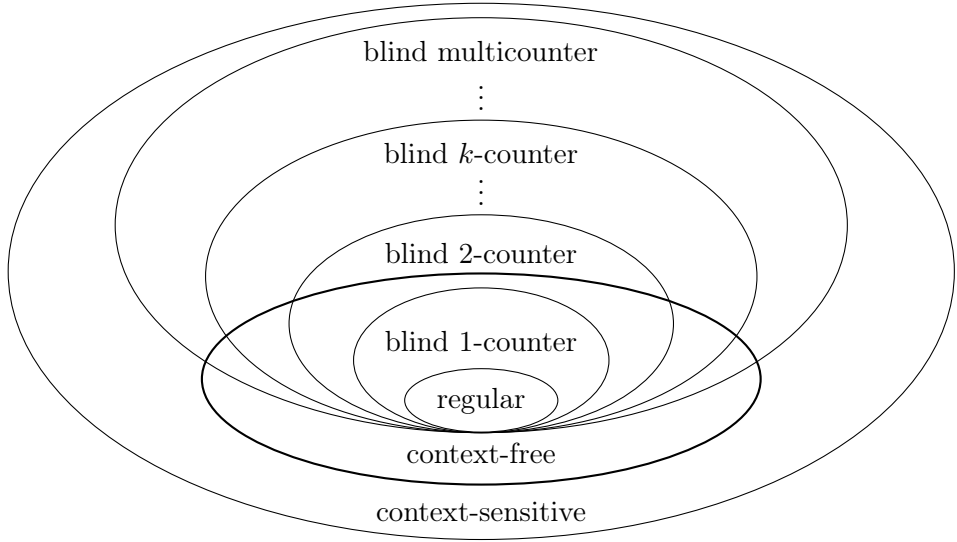

Our definition of blind multicounter automata differs slightly from the one given by Greibach in~\cite{greibach1978}.
In particular, we introduce $\mathfrak{e}$ as an end of input symbol, and allow our automata to add and subtract any constant vector from their counters on a transition instead of only allowing basis vectors.
However, it is clear that this does not increase the expressive power of our model.
Formally, we define a blind $k$-counter automaton follows.

\begin{definition}\label{defn:blind-k-counter-automata}
	Let $k \in \mathbb{N}$, then a \emph{blind $k$-counter machine} is a $6$-tuple of the form $M = (Q,\Sigma,\delta,q_0,F, \mathfrak{e})$ where
	\begin{enumerate}
		\item $Q$ is a finite set of \emph{states};
		\item $\Sigma$ is a finite \emph{input alphabet};
		\item $\delta$ is a finite subset of
		\[
			\big(
				Q
				\times
				(\Sigma \cup \{\varepsilon, \mathfrak{e}\})
			\big)
			\times
			\big(
				Q
				\times
				\mathbb{Z}^k
			\big)
		\]
		which we call the \emph{transition relation};
		\item $q_0 \in Q$ is the \emph{initial state};
		\item $F \subseteq Q$ is the set of \emph{final states}; and
		\item $\mathfrak{e} \notin \Sigma$ is the \emph{end of tape symbol}.
	\end{enumerate}
\end{definition}

Let $M = (Q,\Sigma,\delta,q_0,F, \mathfrak{e})$ be a bind $k$-counter automaton.
Then $M$ begins in state $q_0$ with zero on all its counters.
Suppose that there is a transition relation $((q,a),(p,v)) \in \delta$ with $p,q \in Q$, $a \in \Sigma \cup \{\varepsilon,\mathfrak{e}\}$ and $v \in \mathbb{Z}^k$;
if $M$ is in state $q$ with $a$ as the next letter on its input tape, then it can transition to state $p$ after adding $v$ to its counters and consuming $a$ from its input tape.
The machine accepts if it is in a state of the form $q \in F$ with no letters remaining on its input tape and zero on all its counters.

Formally, we represent the configuration of a blind $k$-counter automaton $M$ with an \emph{instantaneous description} of the form
\[
	(
		q,(c_1,c_2,\ldots,c_k),\sigma\mathfrak{e}
	)
	\in Q \times \mathbb{Z}^k \times \Sigma^*\mathfrak{e}
\]
where $q \in Q$ is the \emph{current state}, $(c_1,c_2,\ldots,c_k) \in \mathbb{Z}^k$ are the values of the counters, and $\sigma \in \Sigma^*$ is the sequence of letters which have yet to be consumed.
Let $C_1$ and $C_2$ be instantaneous descriptions for the configuration of $M$.
Then we write $C_1 \vdash C_2$ if $M$ can move from configuration $C_1$ to $C_2$ in a single transition.
Formally, we interpret the transition relation $\delta$ as follows.

For each relation of the form $((q,s),(p,v)) \in \delta$ with $s \in \Sigma\cup\{\varepsilon\}$, and for each $\sigma = s\sigma' \in \Sigma^*$, we have transitions of the form
\[
	(
		q,(c_1,c_2,\ldots,c_k),\sigma\mathfrak{e}
	)
	\vdash
	(
		p,(c_1+v_1,c_2+v_2,\ldots,c_k+v_k),\sigma'\mathfrak{e}
	).
\]
Moreover, for each relation $((q,\mathfrak{e}),(p,v)) \in \delta$ we have transitions
\[
	(
		q,(c_1,c_2,\ldots,c_k),\mathfrak{e}
	)
	\vdash
	(
		p,(c_1+v_1,c_2+v_2,\ldots,c_k+v_k),\mathfrak{e}
	).
\]
Notice that we do not consume the end of tape symbol $\mathfrak{e}$.

We then write $\vdash^*$ for the transitive symmetric closure of $\vdash$, that is, we have $C_1 \vdash^* C_2$ if $M$ can move from $C_1$ to configuration $C_2$ in finitely many transitions.
We then say that a word $\sigma \in \Sigma^*$ is accepted by $M$ if
\[
	(q_0,(0,0,\ldots,0),\sigma\mathfrak{e})
	\vdash^*
	(q,(0,0,\ldots,0),\mathfrak{e})
\]
for some $q \in F$.
The language of $M$ is defined as
\[
	L(M)
	=
	\left\{
		\sigma \in \Sigma^*
	\mid
		(q_0,(0,0,\ldots,0),\sigma\mathfrak{e})
		\vdash^*
		(q,(0,0,\ldots,0),\mathfrak{e})
		\text{ where }
		q \in F
	\right\},
\]
that is, $L(M)$ is the language of all words accepted by $M$.

\begin{theorem}\label{thm:virtually-abelian-are-blind-counter}
	The language of geodesics of a virtually abelian group with respect to any finite (weighted monoid) generating set $S$ is blind multicounter.
\end{theorem}

\begin{proof}

Let $G$ be a virtually abelian group that is generated as a monoid by a finite weighted set $S$, and let $\mathbb{Z}^n \triangleleft G$ with finite index $d = [G : \mathbb{Z}^n]$.
Let $\sigma \in S^*$, then from \cref{algo:word-shuffling} we have a patterned word $(v,\pi) = \mathrm{Shuffle}(\sigma)$ for which $v^\pi \simeq \sigma$ and thus $\sigma$ is a geodesic if and only if $v \in \mathcal{G}_{\pi}$ where $\mathcal{G}_\pi \subseteq \mathcal{N}_\pi$ is the polyhedral set given by \cref{lemma:geodesics-in-special-form}.

The idea of our proof is to simulate \cref{algo:word-shuffling} on a blind multicounter automaton, while maintaining enough information on the machine's counters so that we may verify the membership of the vector $v$ to the set $\mathcal{G}_{\pi}$.

For each polyhedral set $\mathcal{G}_\pi$, we fix a finite union of basic polyhedral sets
\[
	\mathcal{G}_{\pi}
	=
	\bigcup_{i=1}^{N_\pi}
	\mathcal{B}_{\pi,i}.
\]
Then, for each basic polyhedral set $\mathcal{B}_{\pi,i}$, we fix a finite intersection
\begin{multline}\label{eq:decompose-basic-poly}
	\mathcal{B}_{\pi,i}
	=
	\bigcap_{j=1}^{K_{\pi,i,1}}
		\left\{
			z \in \mathcal{Z}_\pi
		\,\middle\vert\,
			\alpha_{\pi,i,j} \cdot z > \beta_{\pi,i,j}
		\right\}
	\\\cap
	\bigcap_{j=1}^{K_{\pi,i,2}}
		\left\{
			z \in \mathcal{Z}_\pi
		\,\middle\vert\,
			\chi_{\pi,i,j} \cdot z \equiv \eta_{\pi,i,j}\ (\bmod\ \theta_{\pi,i,j})
		\right\}
	\\\cap
	\bigcap_{j=1}^{K_{\pi,i,3}}
		\left\{
			z \in \mathcal{Z}_\pi
		\,\middle\vert\,
			\xi_{\pi,i,j} \cdot z = \lambda_{\pi,i,j}
		\right\}
\end{multline}
where $\alpha_{\pi,i,j},\chi_{\pi,i,j},\xi_{\pi,i,j} \in \mathcal{Z}_\pi$, $\beta_{\pi,i,j},\eta_{\pi,i,j},\lambda_{\pi,i,j} \in \mathbb{Z}$ and $\theta_{\pi,i,j} \in \mathbb{N}_+$.

Let $k \in \mathbb{N}$ be such that $k \geqslant K_{\pi,i,1} + K_{\pi,i,2} + K_{\pi,i,3}$ for each basic polyhedral set $\mathcal{B}_{\pi,i}$.
In the remainder of this proof, we construct a blind $k$-counter automaton $M = (Q,S,\delta,q_0,F,\mathfrak{e})$ that recognises the language of geodesics.
Notice that the input alphabet of the machine is the generating set $S$.

For each basic polyhedral set $\mathcal{B}_{\pi,i}$, we define a map $C_{\pi,i} \colon \mathcal{N}_\pi \to \mathbb{Z}^k$ as
\begin{multline}\label{eq:configuration-of-counters}
	C_{\pi,i}(v) = (
		\alpha_{\pi,i,1} \cdot {v},\,
		\alpha_{\pi,i,2} \cdot {v},\,
		\ldots,\,
		\alpha_{\pi,i,K_{\pi,i,1}} \cdot {v},
	\\
		\chi_{\pi,i,1} \cdot {v},\,
		\chi_{\pi,i,2} \cdot {v},\,
		\ldots,\,
		\chi_{\pi,i,K_{\pi,i,2}} \cdot {v},\,
	\\
		\xi_{\pi,i,1} \cdot {v},\,
		\xi_{\pi,i,2} \cdot {v},\,
		\ldots,\,
		\xi_{\pi,i,K_{\pi,i,3}} \cdot {v},
		0,0,\ldots,0
	).
\end{multline}
Notice that a vector $v \in \mathcal{N}_{\pi}$ belongs to $\mathcal{B}_{\pi,i}$ if and only if
\[
	C_{\pi,i}(v)
	=
	(
		a_1, a_2,\ldots, a_{K_{\pi,i,1}},\,
		b_1, b_2,\ldots, b_{K_{\pi,i,2}},\,
		c_1, c_2,\ldots, c_{K_{\pi,i,3}},\,
		0,0,\ldots,0
	)
\]
where each $a_{j} > \beta_{\pi,i,j}$, each $b_j \equiv \eta_{\pi,i,j}\ (\bmod\ \theta_{\pi,i,j})$ and each $c_{j} = \lambda_{\pi,i,j}$.

\medskip%
\noindent%
\textit{State-Space of the Machine.}%
\nopagebreak\smallskip%
\nopagebreak\par%
\nopagebreak\noindent%
For each $\tau \in \text{\sc Patt}$, each basic polyhedral set $\mathcal{B}_{\pi,i}$, and each word $w \in S^*$ with $|w|_S \leqslant d$, we have a state of the form $[\tau,w,\pi,i] \in Q$.
From these states, the machine will perform \cref{algo:word-shuffling} on its input word.

During the construction of our machine, we will ensure that if
\[
	(q_0,(0,0,\ldots,0),\sigma\mathfrak{e})
	\vdash^*
	([\tau,w,\pi,i],(c_1,c_2,\ldots,c_k),\zeta\mathfrak{e}),
\]
then there is a $u \in \mathcal{N}_\tau$ with $u^\tau w \zeta \simeq \sigma$ and $(c_1,c_2,\ldots,c_k) = C_{\pi,i}(\mathrm{Proj}_{\tau,\pi}(u))$.
In particular, this vector will correspond to some vector $u^{(i)}$ in the sequence of extended patterned words given in (\ref{algo:word-shuffling/sequence}) as constructed in \cref{algo:word-shuffling}.

For each basic polyhedral set $\mathcal{B}_{\pi,i}$, we have an accepting state $q_{\pi,i} \in F$.
Moreover, our construction will have the property that $(q_0,\mathbf{0},\sigma\mathfrak{e})\vdash^*(q_{\pi,i},\mathbf{0},\mathfrak{e})$ if and only if $\mathrm{Shuffle}(\sigma) = (v,\pi)$ with $v \in \mathcal{B}_{\pi,i}$,
and thus the machine $M$ will accept the word $\sigma$ if and only if it is a geodesic.

\medskip%
\noindent%
\textit{Nondeterministically Guessing a Basic Polyhedral Set.}%
\nopagebreak\smallskip%
\nopagebreak\par%
\nopagebreak\noindent%
The machine $M$ begins simulating the word shuffling algorithm after nondeterministically guessing a basic polyhedral set $\mathcal{B}_{\pi,i}$ for which $\mathrm{Shuffle}(\sigma) = (v,\pi)$ with $v \in \mathcal{B}_{\pi,i}$.
Notice that such a choice of basic polyhedral set exists if and only if $\sigma$ is a geodesic.
We accomplish this by introducing a relation
\[
	((q_0,\varepsilon),([\varepsilon,\varepsilon,\pi,i],\mathbf{0})) \in \delta
\]
for each basic polyhedral set $\mathcal{B}_{\pi,i}$,
that is, we have a transition
\begin{equation}\label{eq:first-transition}
	(q_0,(0,0,\ldots,0),\sigma\mathfrak{e})
	\vdash
	([\varepsilon,\varepsilon,\pi,i],(0,0,\ldots,0),\sigma\mathfrak{e})
\end{equation}
for each $\mathcal{B}_{\pi,i}$.
Notice that $\mathbf{0}^\varepsilon\sigma \simeq \sigma$ and $(0,0,\ldots,0) = C_{\pi,i}(\mathrm{Proj}_{\varepsilon,\pi}(\mathbf{0}))$.

\medskip%
\noindent%
\textit{Performing the Word Shuffling Algorithm.}%
\nopagebreak\smallskip%
\nopagebreak\par%
\nopagebreak\noindent%
For each extended strongly patterned word $((u^{(i)},\tau^{(i)}),\sigma^{(i)})$ in sequence (\ref{algo:word-shuffling/sequence}) in \cref{algo:word-shuffling}, we will see that $M$ has configurations of the form
\[
	([w,\tau^{(i)},\pi,i],C_{\pi,i}(u^{(i)}),\zeta)
\]
where $\sigma^{(i)} = w\zeta$.
In order to apply the map $\Delta$ from such a configuration we will require that $w = \mathrm{Prefix}(\sigma)$.
We do so by introducing transitions as follows.

Let $([\tau,w,\pi,i], (c_1, c_2, \ldots, c_k), \zeta\mathfrak{e})$ be a configuration of $M$ and let $\sigma = w \zeta$,
then $w = \mathrm{Prefix}(\sigma)$ if and only if either $|w|_S = d$ or $\zeta = \varepsilon$.
Thus, for each word $w \in S^*$ with $|w|_S < d$ and each $s \in S$, we introduce a relation of the form
\[
	(([\tau,w,\pi,i],s),([\tau,ws,\pi,i],\mathbf{0})) \in \delta
\]
for each $\tau,\pi,i$.
From these relations we have a unique partial computation
\begin{equation}\label{eq:finding-prefix}
	([\tau,w,\pi,i], (c_1, c_2, \ldots, c_k), \zeta\mathfrak{e})
	\vdash^*
	([\tau,w',\pi,i], (c_1, c_2, \ldots, c_k), \zeta'\mathfrak{e})
\end{equation}
where $w' = \mathrm{Prefix}(\sigma)$ and $\sigma = w\zeta = w'\zeta'$.
We then apply the map $\Delta$ as follows.

Let $\tau \in \text{\sc StrPatt}$ be a strong pattern, let $w,\zeta \in S^*$ with $w = \mathrm{Prefix}(w\zeta)$ and $|w|_S \geqslant 1$, and let $\Delta(\tau,w) = (x,\tau',w')$.
From \cref{lemma:map-delta}, we see that for each vector $u \in \mathcal{N}_\tau$ we have $(u')^{\tau'}w'\zeta \simeq u^\tau w \zeta$ where $u' = \mathrm{Proj}_{\tau,\tau'}(u)+e_{\tau',x}$.
Moreover, we see that
\[
	C_{\pi,i}(\mathrm{Proj}_{\tau,\pi}(u')) =
		C_{\pi,i}(\mathrm{Proj}_{\tau,\pi}(u)) +
		C_{\pi,i}(\mathrm{Proj}_{\tau',\pi}(e_{\tau',x})).
\]
Notice that $w = \mathrm{Prefix}(w\zeta)$ if and only if either $|w|_S = d$ or $\zeta = \varepsilon$.
If $|w|_S = d$, then we introduce the relation
\[
	(([\tau,w,\pi,i],\varepsilon),([\tau',w',\pi,i],C_{\pi,i}(\mathrm{Proj}_{\tau',\pi}(e_{\tau',x})))) \in \delta
\]
for each $\pi,i$;
otherwise, if $|w|_S < d$, then we introduce the relation
\[
	(([\tau,w,\pi,i],\mathfrak{e}),([\tau',w',\pi,i],C_{\pi,i}(\mathrm{Proj}_{\tau',\pi}(e_{\tau',x})))) \in \delta
\]
for each $\pi,i$.
That is, we may apply the map $\Delta$ with the above relations.

Combining these transitions with those described in (\ref{eq:finding-prefix}), we see that after nondeterministically choosing a basic polyhedral set in (\ref{eq:first-transition}), the machine will deterministically perform \cref{algo:word-shuffling}, then enter a configuration of the form
\begin{equation}\label{eq:after-word-shuffle}
	(
		[\tau,\varepsilon,\pi,i],(c_1, c_2,\ldots,c_k), \mathfrak{e}
	)
\end{equation}
with $(c_1,c_2,\ldots, c_k) = C_{\pi,i}(v)$ where $(v,\tau) = \mathrm{Shuffle}(\sigma)$.

For each pair of patterns $\pi,\tau$ with $\pi \neq \tau$, and each basic polyhedral set $\mathcal{B}_{\pi,i}$, the machine has no transitions out of any configuration $([\tau,\varepsilon,\pi,i],c,\mathfrak{e})$ where $c \in \mathbb{Z}^k$.
Hence, if the computation enters such a state, it cannot continue to an accepting configuration.
Thus, we may assume without loss of generality that the machine nondeterministically chose the basic polyhedral set $\mathcal{B}_{\pi,i}$ with $\pi = \tau$ when performing the transition in (\ref{eq:first-transition}).
In the rest of our proof, we describe how the machine verifies that $v \in \mathcal{B}_{\pi,i}$.

\medskip%
\noindent%
\textit{Checking Polyhedral Set Membership.}%
\nopagebreak\smallskip%
\nopagebreak\par%
\nopagebreak\noindent%
Suppose that
\[
	(q_0,(0,0,\ldots,0),\sigma\mathfrak{e})
	\vdash^*
	([\pi,\varepsilon,\pi,i],(c_1,c_2,\ldots,c_k),\mathfrak{e}),
\]
then $(c_1,c_2,\ldots,c_k) = C_{\pi,i}(v)$ where $(v,\pi) = \mathrm{Shuffle}(\sigma)$.
For each state of the form $[\pi,\varepsilon,\pi,i]$, we introduce a relation
\[
	(([\pi,\varepsilon,\pi,i],\mathfrak{e}),(q_{\pi,i},\mu_{\pi,i})) \in \delta
\]
where
\begin{multline*}
	\mu_{\pi,i}
	=
	(
		-\beta_{\pi,i,1}-1,
		-\beta_{\pi,i,2}-1,
		\ldots,
		-\beta_{\pi,i,k}-1,
	\\
		-\eta_{\pi,i,1},
		-\eta_{\pi,i,2},
		\ldots
		-\eta_{\pi,i,k},
	\\
		-\lambda_{\pi,i,1},
		-\lambda_{\pi,i,2},
		\ldots
		-\lambda_{\pi,i,k},
		0,0,\ldots,0
	).
\end{multline*}
From this relation we have
\[
	([\pi,\varepsilon,\pi,i],(c_1,c_2,\ldots,c_k),\mathfrak{e})
	\vdash
	(q_{\pi,i},(c'_1,c'_2,\ldots,c'_k),\mathfrak{e})
\]
where $v \in \mathcal{B}_{\pi,i}$ if and only if $(c'_1,c'_2,\ldots,c'_k)$ belongs to the set
\[
	\mathbb{N}^{K_{\pi,i,1}} \times 
	\theta_{\pi,i,1}\mathbb{Z}
	\times \theta_{\pi,i,2}\mathbb{Z}
	\times \cdots
	\times \theta_{\pi,i,K_{\pi,i,2}}\mathbb{Z}
	\times \{0\}^{k-K_{\pi,i,1}-K_{\pi,i,2}}.
\]
We verify $v$'s membership to $\mathcal{B}_{\pi,i}$ by introducing additional relations as follows.
For each $1 \leqslant j \leqslant K_{\pi,i,1}$, we have
\[
	((q_{\pi,i},\mathfrak{e}),(q_{\pi,i},-e_j)) \in \delta,
\]
where $e_j \in \mathbb{Z}^k$ is the $j$-th standard basis element, and for each $ 1 \leqslant j \leqslant  K_{\pi,i,2}$
\[
	((q_{\pi,i},\mathfrak{e}),(q_{\pi,i},\pm\theta_{\pi,i,j} \, e_{j'})) \in \delta
\]
where $j' = K_{\pi,i,1}+j$ and $e_{j'} \in \mathbb{Z}^k$ is the $j'$-th standard basis element.
From these relations, we see that we see that
\[
	(q_0,(0,0,\ldots,0),\sigma\mathfrak{e})
	\vdash^*
	(q_{\pi,i},(0,0,\ldots,0),\mathfrak{e}),
\]
if and only if $v \in \mathcal{B}_{\pi,i}$ where $(v,\pi) = \mathrm{Shuffle}(\sigma)$.
\end{proof}

\section*{Acknowledgements}

The author would like to thank their supervisor Murray Elder for introducing this interesting research topic, and the anonymous reviewer for their detailed and helpful comments and suggestions.
This research is supported by an Australian Government Research Training Program Scholarship, and partially supported by Australian Research Council grant DP160100486.

% \bib, bibdiv, biblist are defined by the amsrefs package.

\end{document}